\documentclass[11pt,final]{article} % use larger type; default would be 10pt
\usepackage{etex}
\usepackage[utf8]{inputenc} % set input encoding (not needed with XeLaTeX)
\def\bibsection{\section*{References}}

\usepackage{geometry} % to change the page dimensions
\usepackage{graphicx} % support the \includegraphics command and options

\usepackage{booktabs} % for much better looking tables
\usepackage{array} % for better arrays (eg matrices) in maths
\usepackage{paralist} % very flexible & customisable lists (eg. enumerate/itemize, etc.)
\usepackage{verbatim} % adds environment for commenting out blocks of text & for better verbatim
\usepackage{subfig} % make it possible to include more than one captioned figure/table in a single float
\usepackage{natbib}
\usepackage{hyperref}

\usepackage{fancyhdr} % This should be set AFTER setting up the page geometry
\usepackage{sectsty}
\allsectionsfont{\sffamily\mdseries\upshape} % (See the fntguide.pdf for font help)
\usepackage[nottoc,notlof,notlot]{tocbibind} % Put the bibliography in the ToC
\usepackage[titles,subfigure]{tocloft} % Alter the style of the Table of Contents

\usepackage{amsmath}
\usepackage{amssymb}
\usepackage{amsthm}
\usepackage{mathrsfs}
\usepackage{enumerate}
\usepackage[all]{xy}
\usepackage{todonotes}
\usetikzlibrary{calc}

\numberwithin{equation}{section}
\newtheorem{thm}[equation]{Theorem}
\newtheorem{prop}[equation]{Proposition}
\newtheorem{cor}[equation]{Corollary}

\newtheorem{lem}[equation]{Lemma}
\theoremstyle{definition}
\newtheorem{prob}[equation]{Problem}

\newtheorem{defn}[equation]{Definition}
\newtheorem{ex}[equation]{Example}

\newtheorem{remk}[equation]{Remark}

\newcommand{\tn}[1]{\textnormal{#1}}
\newcommand{\ol}[1]{\overline{#1}}
\newcommand{\wt}[1]{\widetilde{#1}}

\newcommand{\vocab}[1]{\textbf{#1}}

\newcommand{\bbr}{\mathbb R}
\newcommand{\bbz}{\mathbb Z}

\newcommand{\bbn}{\mathbb N}

\newcommand{\bbc}{\mathbb C}

\newcommand{\vspan}{\tn{span}}

\newcommand{\diff}{\backslash}

\newcommand{\diag}{\text{diag}}

\newcommand{\ann}{\tn{Ann}}

\newcommand{\tr}{\textnormal{tr}}

\renewcommand{\Re}{\text{Re}}
\renewcommand{\Im}{\text{Im}}

\newcommand{\spec}{\textnormal{Spec}}

\newcommand{\supp}{\textnormal{supp}}

\newcommand{\fourl}{\mathcal{F}_L}
\newcommand{\fourr}{\mathcal{F}_R}
\newcommand{\bispl}{\mathcal{B}_L}
\newcommand{\bispr}{\mathcal{B}_R}
\newcommand{\shift}{\mathcal{S}}

\newcommand{\weyll}{\Omega[[x]]}

\newcommand{\mweyll}{M_N(\weyll)}
\newcommand{\Ad}{\textnormal{Ad}}
\newcommand{\mxx}[4]{\left(\begin{array}{cc} #1 & #2\\ #3 & #4\end{array}\right)}

\title{\bf{The Matrix Bochner Problem}}
\author{W. Riley Casper\thanks{Email:\href{wcasper1@lsu.edu}{wcasper1@lsu.edu}} \ and Milen Yakimov\thanks{Email:\href{yakimov@math.lsu.edu}{yakimov@math.lsu.edu}}
\\ \hfill\\ \normalsize \textit{Department of Mathematics,
Louisiana State University,} \\
\textit{Baton Rouge, LA 70803,
USA 
}     }

\date{} % Activate to display a given date or no date (if empty),
\begin{document}

\maketitle
\begin{abstract}
A long standing question in the theory of orthogonal matrix polynomials is the matrix Bochner problem, the classification of $N \times N$ 
weight matrices $W(x)$ whose associated orthogonal polynomials are eigenfunctions of a second order differential operator. 
Based on techniques from noncommutative algebra (semiprime PI algebras of Gelfand-Kirillov dimension one),
we construct a framework for the systematic study of the structure of the algebra $\mathcal D(W)$ of matrix differential operators 
for which the orthogonal polynomials of the weight matrix $W(x)$ are eigenfunctions. The ingredients for this algebraic setting are derived from the 
analytic properties of the orthogonal matrix polynomials. We use the representation theory of the algebras $\mathcal D(W)$ to resolve the matrix Bochner problem 
under the two natural assumptions that the sum of the sizes of the matrix algebras in the central localization of $\mathcal D(W)$ equals $N$ (fullness of $\mathcal D(W)$) 
and the leading coefficient of the second order differential operator multiplied by the weight $W(x)$ is positive definite. In the case of $2\times 2$ weights, it is proved that 
fullness is satisfied as long as $\mathcal D(W)$ is noncommutative. The two conditions are natural in that without them the problem 
is equivalent to much more general ones by artificially increasing the size of the matrix $W(x)$.
\end{abstract}

\tableofcontents
\section{Introduction}
\subsection{An overview of the results in the paper}
Orthogonal matrix polynomials are sequences of matrix-valued polynomials which are pairwise orthogonal with respect to a matrix-valued inner product defined by a weight matrix $W(x)$.
They were defined seventy years ago by Krein \cite{krein} and since then have been shown to have a wide variety of applications in both pure and applied mathematics, including spectral theory, quasi-birth and death processes, signal processing, Gaussian quadrature, special functions, random matrices, integrable systems and representation theory.
As natural generalizations of their classical scalar counterparts, orthogonal matrix polynomials have been shown to have a wide variety of applications in both pure and applied mathematics, including spectral theory, quasi-birth and death processes, signal processing, Gaussian quadrature, special functions, random matrices, integrable systems and representation theory.
Of great specific interest are those orthogonal matrix polynomials which are simultaneously eigenfunctions of a matrix-valued differential operator.
They in particular generalize the classical orthogonal polynomials of Hermite, Laguerre, and Jacobi whose utility in diverse research areas is difficult to understate.

The current and potential applications of matrix-valued orthogonal polynomials and the study of their analytic properties
naturally motivate the problem of the classification of all orthogonal matrix polynomials which are eigenfunctions of a second-order differential operator.
This problem was posed and solved by Bochner \cite{bochner1929sturm} in the scalar case and later extended by Dur\'{a}n \cite{duran1997} to the matrix case.
\begin{prob}[Matrix Bochner Problem]\label{matrix bochner problem}
Classify all $N\times N$ weight matrices $W(x)$ whose associated sequence of orthogonal matrix polynomials are eigenfunctions of a second-order matrix differential operator.
\end{prob}
Bochner \cite{bochner1929sturm}
proved that for $N=1$ up to an affine change of coordinates the only weight matrices satisfying these properties are 
the classical weights $e^{-x^2}$, $x^be^{-x}1_{(0,\infty)}(x)$, and $(1-x)^a(1+x)^b1_{(-1,1)}(x)$ of Hermite, Laguerre, and Jacobi respectively.
For brevity, we hereafter refer to affine transformations of these weight functions as classical weights.
However, for $N>1$ the solution of the matrix Bochner problem has proved difficult.

Gr\"unbaum, Pacharoni and Tirao \cite{grunbaum-HC1,grunbaum-HC2} 
found the first nontrivial solutions of the matrix Bochner problem using Harish-Chandra modules for real simple groups and the associated 
matrix spherical functions.
In the past twenty years numerous other examples have been found. 
%\cite{duran2004, duran2005c, grunbaum2005} to name a few. 
 More recent work has focused on the study of the algebra $\mathcal{D}(W)$ of all differential operators for which the matrix-valued polynomials are eigenfunctions.  Alternatively, 
 $\mathcal D(W)$ may be described as the algebra of all differential operators which are degree-preserving and $W$-adjointable.  This algebra is studied both from a general stand point and for specific weights $W(x)$, 
\cite{castro2006,grunbaum2007b,tirao2011,zurrian2016,zurrian2016algebra}.  
However, the previous general results on the structure of $\mathcal D(W)$ are very limited 
and general classification results on the matrix Bochner problem have remained illusive.

In this paper, we carry out a general study of the algebra $\mathcal{D}(W)$ using techniques from noncommutative algebra.
Starting from the analytics properties of the related sequence of orthogonal matrix polynomials, we equip the algebra 
$\mathcal{D}(W)$ with the structure of a $*$-algebra with a positive involution. From here we obtain that 
$\mathcal{D}(W)$ is an affine semiprime algebra of Gelfand-Kirillov dimension at most one.  Among the hardest of these properties to prove is that $\mathcal{D}(W)$ is affine.  We show this by first proving that the center is affine by constructing an embedding into a product of Dedekind domains.
We then show that the whole algebra $\mathcal D(W)$ is a large subalgebra of an order in a semisimple algebra.
Once all the properties of $\mathcal D(W)$ are established, the Small-Stafford-Warfield theorem \cite{small1985} tells us $\mathcal{D}(W)$ is a Noetherian algebra which is finitely generated module over its center $\mathcal Z (W)$.
This is used to show that the the extension of scalars of $\mathcal{D}(W)$ to the total ring of fractions of $\mathcal Z (W)$
is isomorphic to a product of matrix algebras over the function fields of the irreducible components 
of $\mathrm{Spec} \  \mathcal Z (W)$ (one matrix algebra for each irreducible component). 
This fact in turn allows us to define the notion of the \emph{rank of the algebra} $\mathcal{D}(W)$ which 
equals to the sum of the sizes of the matrix algebras. This integer is shown to be between $1$ and $N$ 
and another characterizations of it is given as the maximal number of generalized orthogonal idempotents of $\mathcal{D}(W)$ 
which sum to a central element. 

The above structural results allow us to define canonical submodules of the $\Omega(x), \mathcal{D}(W)$-bimodule $\Omega(x)^{\oplus N}$, 
where $\Omega(x)$ is the matrix Weyl algebra with rational coefficients. 
Using representation theory, we demonstrate that the algebraic structure of the algebra $\mathcal{D}(W)$ 
has a profound influence on the shape of the weight matrix $W(x)$ itself. Specifically, using  $\mathcal{D}(W)$ modules defined from the maximal set of orthogonal idempotents, 
we prove that when the algebra $\mathcal{D}(W)$ is \emph{full} in the sense that the rank is as large as possible (i.e., equals $N$),
the matrix $W(x)$ is congruent to a diagonal weight matrix via a rational matrix $T(x)$,
\begin{equation}\label{structural equation}
W(x) = T(x)\diag(f_1(x),\dots, f_N(x))T(x)^*
\end{equation}
where $f_i(x)$ is a classical weight for all $i$. Further arguments with these $\mathcal{D}(W)$ modules allow us to control the size of the \emph{Fourier} algebra 
of the matrix weight $W(x)$, which is defined as the algebra of  matrix differential operators that applied to the orthogonal matrix polynomials of $W(x)$ equal to 
a shift operator applied to the same sequence. This is a larger algebra than $\mathcal{D}(W)$, which in turn is used to show that under natural assumptions, solutions 
of the matrix Bochner problem come from bispectral Darboux transformations  of a direct sum of classical weights.

Our main theorem is the following Bochner-type classification result:
\begin{thm}[Classification Theorem]\label{classification theorem}
\label{thm1.3}
Let $W(x)$ be a weight matrix and suppose that $\mathcal D(W)$ contains a $W$-symmetric second-order differential operator
\begin{equation}
\label{2nd-order}
\mathfrak D = \partial_x^2D_2(x) + \partial_xD_1(x) + D_0(x)
\end{equation}
with $D_2(x)W(x)$ positive-definite on the support of $W(x)$.
Then the algebra $\mathcal D(W)$ is full if and only if $W(x)$ is a noncommutative bispectral Darboux transformation of a direct sum of classical weights.
Furthermore, in this case (\ref{structural equation}) holds.
\end{thm}

A matrix differential operator is called $W$-symmetric when it equals its formal adjoint with respect to $W(x)$, see Definition \ref{formal adjoint}. 
The algebra $\mathcal D(W)$ contains a second-order differential operator if and only if it contains a $W$-symmetric second-order differential operator. 
For such an operator (\ref{2nd-order}), the matrix $D_2(x)W(x)$ is necessarily Hermitian.

We pause here to briefly explain the assumptions in the theorem, arguing that they are both relatively weak and natural.  The requirement that $\mathcal D(W)$ contains a second-order $W$-symmetric operator $\mathfrak D$ with $D_2(x)W(x)$ positive definite on the support of $W(x)$ is a non-degeneracy condition.  It implies that $\mathfrak D$ is not a zero divisor in $\mathcal D(W)$ and that $\mathfrak D$ does not annihilate an infinite dimensional space of matrix-valued polynomials.
It fits well into the literature wherein the stronger condition that the leading coefficient of the second-order differential operator is a scalar is often assumed.

Furthermore, without this positivity assumption, the problem is closely related to the much more general one for classifying matrix weights for which $\mathcal D(W)$ is nontrivial, i.e., 
contains a differential operator of an arbitrary 
nonzero order. This problem is not solved even in the scalar case. To see the stated relation, consider an $N \times N$ matrix weight $W(x)$ for which $\mathcal D(W)$ 
contains an operator of order $2k$. Let $\wt W(x)$ be a bispectral Darboux transformation of $W(x)$ 
obtained by factorizing the differential operator into a product of operators of orders $k$ in the sense of Definition \ref{bisp-D}.  
Then the $2N \times 2N$ block diagonal weight matrix $\diag(W(x), \wt W(x))$ has the property that $\mathcal D(\diag(W, \wt W))$ contains a differential 
operator of order $k$, see Example \ref{double-size} for details.
In various common situations this can be iterated to bring the minimal order of a differential operator in  $\mathcal D(W)$ down to 2 by increasing the size of $W(x)$. 
The positivity assumption on the leading terms of the 
differential operator $\mathfrak D$ avoids this vast expansion of the matrix Bochner problem. In a similar fashion, when 
one artificially increases the size of a matrix weight, the difference between $N$ and the rank of $\mathcal D(W)$ increases too, leading to cases where 
the algebra  $\mathcal D(W)$ is substantially smaller than the size of the matrix weight $W(x)$.
The fullness condition ensures that the problem does not become wild in this way.

Additionally, under the assumptions of the Classification Theorem we obtain an explicit description of the algebra $\mathcal D(W)$ itself.
In particular Theorem \ref{backup strats} provides us with a matrix differential operator conjugating $\mathcal D(W)$ into a subalgebra of $\mathcal D(f_1\oplus\dots\oplus f_N)$ for some classical weights $f_1,\dots, f_N$.
When the $f_1,\dots,f_N$ are suitably chosen, this latter algebra is a maximal order (in this context a direct product of matrix algebras over polynomial rings).
In this way, we can think of our process as a noncommutative desingularization of the original algebra $\mathcal D(W)$.

We prove that our results take on a more explicit form in the case of matrix weights $W(x)$ of size $2 \times 2$. In this case we show that, if 
$\mathcal D(W)$ is noncommutative, then $\mathcal D(W)$ must be full. This leads us to the following classification result for $2\times 2$ weight matrices.

\begin{thm}[$2\times 2$ case]\label{2x2 classification}
Let $W(x)$ be a $2\times 2$ weight matrix and suppose that $\mathcal D(W)$ contains a $W$-symmetric second-order differential operator whose leading coefficient multiplied by $W(x)$ is positive definite on the support of $W(x)$.
The algebra $\mathcal D(W)$ is noncommutative if and only if the weight $W(x)$ is a noncommutative bispectral Darboux transformation of $r(x)I$ for some classical weight $r(x)$.
\end{thm}
Thus aside from various degenerate cases when $\mathcal D(W)$ consists of polynomials of a single differential operator of order $2$, this fully 
resolves the Bochner problem for the $2\times 2$ case.

For a comprehensive background on  noncommutative rings and polynomial identity (PI) algebras we refer the reader to the books 
\cite{GW,mcconnell2001} and \cite{DF}, respectively. A comprehensive treatment of orders in central simple algebras is 
given in \cite{reiner}. 
A concise and illuminating treatment of PI rings can be found in \cite[Sections I.13 and III.1]{BG}. 

\subsection{From algebraic geometry to noncommutative algebra}
The idea that the algebraic properties of a commutative algebra of differential or difference operators can tell us about the operators themselves is not new.
It is a powerful point of view that can be pointed to as the central concept in the unification of various phenomena in integrable systems and algebraic geometry during the 70s and 80s, \cite{dubrovin,krichever1977,segal1985,vanmoerbeke}.
This resulted in a number of strong applications of algebraic geometry to the analysis of solutions of integrable systems, and 
in spectacular applications in the opposite direction, for instance the solution of the Schottky problem \cite{shiota}. 

More formally, there is a natural correspondence between differential operator algebras, vector bundles on algebraic curves, and exact solutions of certain nonlinear partial differential equations.  The most basic example of this is when a Schr\"odinger operator $\partial^2+u(x)$ commutes with a differential operator of order $3$.  In this case, $u(x)$ gives rise to a solition solution of the Korteweg-del Vries equation \cite{lax1986}
$$\phi_{xxx}(x,t) + 6\phi(x,t)\phi_x(x,t) = \phi_t(x,t).$$

This concept has also been applied in context of bispectral algebras of differential operators of low rank (here rank means the gcd of the orders of differential operators in the algebra).
Most strikingly, Wilson provided a complete classification of bispectral differential operators of rank $1$ in terms of rational projective curves with no nodal singularities \cite{wilson1998}.
General methods for constructing bispectral operators with $C(\mathfrak d)$ of arbitrary rank have been developed 
\cite{bakalov1996,bakalov1997,duistermaat1986,kasman1997} and the bispectral algebras $C(\mathfrak d)$ containing an operator of prime 
order have been classified \cite{horozov}. Moving into a noncommutative direction, rank 1 bispectral solutions of the KP hierarchy 
and the related Calogero--Moser systems were analyzed 
in \cite{BZ-N,BGK,B-W} using $\mathcal D$-modules, one-sided ideals of the first Weyl algebras and noncommutative 
projective surfaces.

The basis of the applications of algebraic geometry to integrable systems is the consideration of 
commutative algebras of differential and difference operators.
New in this paper is the systematic study of the properties of orthogonal matrix polynomials for a weight matrix $W(x)$, 
based on structural results for the algebra $\mathcal D(W)$ which is generally noncommutative.
The noncommutative case is quite a bit more challenging;
for example, we no longer have the power of algebraic geometry to rely on, at least directly. Our methods use 
PI algebras, and more precisely, noncommutative algebras which are module-finite over their centers, their 
relation to central simple algebras, and their representation theory \cite{BG,DF}.

Similarly, to Shiota's solution of the Schottky problem \cite{shiota}, we expect that our methods will have applications of orthogonal matrix polynomials 
to noncommutative algebra, in that interesting PI algebras can be realized as the algebras $\mathcal D(W)$ 
for some weight matrices $W(x)$. Such a relation, can be used to study fine properties of these algebras 
using orthogonal polynomials.
\subsection{Notation}
Throughout this paper, we will use capital letters to represent matrices and lower case letters to represent scalars.
We will also use the Gothic font to represent differential operators and script to represent difference operators.
For example, we will use $f(x)$ or $F(x)$ to represent a function of $x$, depending on whether it is a scalar or matrix-valued function.
Similarly, we will use $\mathfrak d$ or $\mathfrak D$ to represent a differential operator, again depending on whether it is scalar or matrix-valued.
Furthermore an expression like $\mathscr M$ will denote a matrix-valued discrete operator.
Wherever feasible, we will use capitalized calligraphic font, such as $\mathcal A$, to denote various operator algebras, subalgebras, and ideals.
Exceptions to this will include certain special algebras, such as the algebra of all differential operators, that will have their own special notation.

For reasons pertaining to compatibility with the matrix-valued inner product $\langle\cdot,\cdot\rangle_W$ defined below, our differential operators will act on the right.
For example, the basic differential operator $\partial_x$ acts on a function $f(x)$ by
$$f(x)\cdot\partial_x = f'(x).$$
An arbitrary matrix differential operator $\mathfrak D = \sum_{j=0}^n \partial_x^j A_j(x)$ acts on a matrix-valued function $F(x)$ by
$$F(x)\cdot\mathfrak D = \sum_{j=0}^n F^{(j)}(x)A_j(x).$$
Because this action is a right action, the algebra of all differential operators will satisfy the fundamental commutation relation
$$x\partial -\partial x = 1.$$
Note this is reversed from the typical identity for the Weyl algebra, since differential operators are most often taken to act on the left.
Thus our algebra of differential operators will actually be the opposite algebra of usual Weyl algebra.

In addition, we will adopt the following notation
\begin{itemize}
\item
For any ring $R$, $R[x]$,  $R(x)$, $R[[x]]$ and $R((x))$ will denote the rings of polynomials, rational functions, power series and Laurent series with coefficients in $R$, respectively.
\item
We will use $\Omega[x]$, $\Omega(x)$, $\Omega[[x]]$ and $\Omega((x))$ to denote the ring of differential operators with right action and with coefficients in $\bbc[x]$, $\bbc(x)$, $\bbc[[x]]$ and $\bbc((x))$ respectively.  
\item
The symbol $\dag$ will always denote the $W$-adjoint, as defined below.
\item 
For any ring $R$, $M_N(R)$ will denote the ring of matrices with coefficients in $R$ and $E_{ij}$ will denote the element of this ring with a $1$ in the $i,j$'th entry and zeros elsewhere.
\item
The symbols $x$, $t$, and $n$ will represent indeterminants, unless specified otherwise.
\end{itemize}

\section{Background}
\subsection{Orthogonal matrix polynomials and Bochner's problem}\label{matrix orthogonal polynomials}
We begin with a brief review of the basic theory of orthogonal matrix polynomials and Bochner's problem.
\begin{defn}\label{weight matrix definition}
A \vocab{weight matrix $W(x)$} supported on an interval $(x_0,x_1)$ is Hermitian matrix-valued function $W: \bbc\rightarrow M_N(\bbc)$ which is entrywise-smooth on $\bbr$, identically zero outside of $(x_0,x_1)$, and positive definite on $(x_0,x_1)$ with finite moments:
$$\int x^m W(x)dx < \infty,\ \ \forall m\geq 0.$$
We call the interval $(x_0,x_1)$ the \vocab{support of $W(x)$}.
\end{defn}
\begin{remk}
We are restricting our attention to ``smooth" weight matrices in order to avoid more delicate analytic considerations.
In general one may consider matrix valued measures on $\bbr$ satisfying an appropriate generalization of the above definition.  This is discussed further in \cite{damanik2008}.
\end{remk}
A weight matrix $W(x)$ defines a matrix-valued inner product $\langle \cdot,\cdot\rangle_W$ on the vector space $M_N(\bbc[x])$ of all $N\times N$ complex matrix polynomials by
\begin{equation}\label{inner product}
\langle P,Q\rangle_W := \int_\bbc P(x)W(x)Q(x)^* dx,\ \ \forall P,Q\in M_N(\bbc[x]).
\end{equation}

By applying Gram-Schmidt we may determine a sequence of matrix-valued polynomials $P(x,n)\in M_N(\bbc[x])$, $n\in\bbn$, with $P(x,n)$ degree $n$ with nonsingular leading coefficient such that $\langle P(x,n),P(x,m)\rangle_W = 0$ for $m\neq n$.  Moreover this sequence $P(x,n)$ is unique up to the choice of the leading coefficients of the polynomials.  This leads to the following definition.
\begin{defn}
We call a sequence of matrix polynomials $P(x,n)$, $n\in\bbn$, a \vocab{sequence of orthogonal matrix polynomials for $W(x)$} if for all $n$ the polynomial $P(x,n)$ has degree $n$ with nonsingular leading coefficient and 
$$\langle P(x,n),P(x,m)\rangle_W = 0,\ \ \forall m\neq n.$$
The sequence $P(x,n)$ will be called \vocab{monic} if the leading coefficient of each $P(x,n)$ is $I$.
\end{defn}

The sequence of monic orthogonal polynomials of a weight matrix $W(x)$ must satisfy a three-term recursion relation.
Conversely, any sequence of monic orthogonal polynomial satisfying a sufficiently nice three-term recurrence relation will be a sequence of monic orthogonal polynomials for a weight matrix.
\begin{thm}[Dur\'an, Van~Assche and L\'opez-Rodriguez \cite{duran1995,duran2004x}]
Suppose that $P(x,n)$ is a sequence of monic orthogonal matrix polynomials for a weight matrix $W(x)$.  
Then for some sequences of complex matrices $B(n)$ and $C(n)$, we have
\begin{equation}\label{matrix recurrence relation}
xP(x,n) = P(x,n+1) + B(n)P(x,n) + C(n)P(x,n-1),\ \ \forall n\geq 1.
\end{equation}
Conversely, given sequences of matrices $B,C: \bbn\rightarrow\bbc$ satisfying natural mild assumptions, there exists a weight matrix $W(x)$ for which the sequence of polynomials $P(x,n)$ defined by \eqref{matrix recurrence relation} is a sequence of monic orthogonal matrix polynomials.
\end{thm}

The focus of our paper is the Matrix Bochner Problem \ref{matrix bochner problem}.
Specifically, we wish to determine for which weights $W(x)$ the associated sequence $P(x,n)$ of monic orthogonal polynomials satisfy a second order matrix differential equation
\begin{equation}\label{basic eigenvalue equation}
P''(x,n)A_2(x) + P'(x,n)A_1(x) + P(x,n)A_0(x) = \Lambda(n)P(x,n)
\end{equation}
for all $n$ for some sequence of complex matrices $\Lambda(n)$.
Equivalently, we want to know when the $P(x,n)$ are eigenfunctions of a second order matrix differential operator $\mathfrak D = \partial_x^2 A_2(x) + \partial_x A_1(x) + A_0(x)$ acting on the right with matrix-valued eigenvalues.
More generally, we can consider the algebra of all differential operators for which the sequence $P(x,n)$ are eigenfunctions.
\begin{defn}\label{D(W) definition}
Let $W(x)$ be a weight matrix with sequence of monic orthogonal polynomials $P(x,n)$.
We define $\mathcal{D}(W)$ to be the collection of all matrix-valued differential operators for which the $P(x,n)$ are eigenfunctions for all $n$.
We call $\mathcal D(W)$ the \vocab{algebra of matrix differential operators associated to $W(x)$}.
\end{defn}
\begin{remk}
Later we will realize $\mathcal D(W)$ as the right bispectral algebra associated to the bispectral function $P(x,n)$.
However we will retain the notation $\mathcal D(W)$ throughout the paper.
\end{remk}

\subsection{Adjoints of Differential Operators}
Our methods extensively use $*$-algebras constructed using the adjoint of a matrix differential operator.
We briefly recall some definitions and ideas here and refer the reader to the excellent reference \cite{dunfordschwartzII} for a comprehensive treatment.
A matrix differential operator $\mathfrak D$ may always be viewed as an unbounded linear operator on a suitably chosen Hilbert space.
Therefore it makes sense to think about the adjoint of a differential operator strictly in terms of functional analysis.
\begin{defn}
An \vocab{unbounded linear operator} on a Hilbert space $\mathcal H$ is a linear function $T$ defined on a dense subset $\mathcal H$ called the \vocab{domain of $T$}.
The adjoint of $T$ is an unbounded linear operator $T^*$ with domain
$$\{y\in \mathcal H: x\rightarrow \langle Tx,y\rangle\ \text{is continuous}\}$$
defined by $\langle Tx,y\rangle = \langle x,T^*y\rangle$ for all $x$ and $y$ in the domain of $T$ and $T^*$, respectively.
\end{defn}
\begin{remk}
Some authors do not require an unbounded operator to be defined on a dense subspace.
However it is necessary for Hahn-Banach to imply the existence of $T^*$, so we adopt it as part of our definition.
Since we will be working with differential operators with rational coefficients, this definition will be sufficient for us.
\end{remk}

The algebra of differential operators also have a natural adjoint operation $*$ called the formal adjoint.
\begin{defn}
The \vocab{formal adjoint} on $M_N(\Omega((x)))$ is the unique involution extending Hermitian conjugate on $\Omega((x))$ and sending $\partial_xI$ to $-\partial_xI$.
\end{defn}
Now consider specifically the Hilbert space $\mathcal H = L^2([-1,1])$.
The formal adjoint is defined in such a way that the adjoint of an unbounded operator and its formal adjoint will agree on a subset of $\mathcal H$ whose closure has finite codimension.
However as the next example shows, the formal adjoint may not be equal to the adjoint of $\mathfrak D$ as an unbounded linear operator, as the next example shows.
\begin{ex}
Let $\mathcal H= L^2([-1,1])$ with the usual inner product and consider two polynomials $p(x),q(x)$.
Then the formal adjoint of the differential operator $\partial_x$ is $\partial_x^* = -\partial_x$.
However we calculate
\begin{align*}
\langle p(x)\cdot\partial_x,q(x)\rangle
  & = \int_{-1}^1 p'(x)\ol{q(x)} dx\\
  & = p(1)\ol{q(1)}-p(0)\ol{q(0)} - \int_{-1}^1 p(x)\ol{q'(x)}\\
  & = \langle p(x),q(x)\cdot\partial_x^*\rangle + p(1)\ol{q(1)}-p(0)\ol{q(0)}.
\end{align*}
Since there is an extra term on the right hand side, the formal adjoint of $\partial_x$ does not agree with the adjoint as an unbounded linear operator on its domain.
\end{ex}

More generally a weight matrix $W(x)$ defines a Hilbert space of matrix-valued functions $M_N(\mathcal H)$, where $\mathcal H$ is the Hilbert space of complex-valued $L^2(\tr(W(x))dx)$ functions on the support $(x_0,x_1)$ of $W(x)$.
Any matrix-valued differential operator $\mathfrak D\in \mweyll$ with suitably nice coefficients will define an unbounded linear operator on $M_N(\mathcal H)$.
There is also a natural choice of formal adjoint here which takes into account the form of the inner product.
\begin{defn}\label{formal adjoint}
Let $W(x)$ be a weight matrix supported on an interval $(x_0,x_1)$ containing $0$ and let $\mathfrak{D}\in \mweyll$.  A \vocab{formal adjoint of $\mathfrak{D}$ with respect to $W(x)$}, or \vocab{formal $W(x)$-adjoint of $\mathfrak{D}$} is the unique differential operator $\mathfrak{D}^\dag\in \mweyll$ defined on $(x_0,x_1)$ by
$$\mathfrak{D}^\dag := W(x)\mathfrak{D}^* W(x)^{-1}.$$
An operator $\mathfrak{D}$ is called \vocab{formally $W(x)$-symmetric} if $\mathfrak{D}^\dag = \mathfrak{D}$.
\end{defn}

\begin{defn}
Let $W(x)$ be a weight matrix and let $\mathfrak{D}\in M_N(\Omega[x])$.
We say that $\mathfrak D$ is \vocab{$W$-adjointable} if there exists a differential operator $\wt{\mathfrak D}\in M_N(\Omega[x])$ satisfying
$$\langle P\cdot\mathfrak{D},Q\rangle_W = \langle P,Q\cdot\wt{\mathfrak{D}}\rangle_W, \ \ \forall P,Q\in M_N(\bbc[x]).$$
In this case, we call $\wt{\mathfrak D}$ the \vocab{adjoint of $\mathfrak{D}$ with respect to $W(x)$}, or alternatively the \vocab{$W$-adjoint of $\mathfrak{D}$}.
If $\mathfrak{D} = \wt{\mathfrak{D}}$, then $\mathfrak{D}$ is called \vocab{$W$-symmetric}.
\end{defn}

\begin{prop}
Let $W(x)$ be a weight matrix and let $\mathfrak{D}\in\mweyll$ be an amenable matrix differential operator.  If $\mathfrak D$ is $W$-adjointable, then its $W$-adjoint is equal to the formal $W$-adjoint $\mathfrak{D}^\dag$.
\end{prop}

\subsection{Bispectrality}
The notion of bispectrality arose from questions originally posed by Duistermaat and Gr\"{u}nbaum \cite{duistermaat1986} on finding locally meromorphic functions $\psi(x,y)$ which define a family of eigenfunctions for a fixed differential operator in $x$ \emph{and} a family of eigenfunctions of another differential operator in $y$.  Examples of functions $\psi(x,y)$ satisfying this property are referred to in the literature as bispectral functions.  The simplest example of a bispectral function is the exponential function $\psi(x,y) = e^{xy}$.
\begin{ex}
The exponential function $\psi(x,y) = e^{xy}$ satisfies
$$\partial_x\cdot\psi(x,y) = y\psi(x,y),\ \text{and}\ \partial_y\cdot\psi(x,y) = x\psi(x,y).$$
Therefore $\psi(x,y)$ defines a bispectral function.
\end{ex}

Since its inception, the concept of bispectrality has been generalized in various natural directions, particularly in directions which include certain noncommutative aspects.
To begin, we present bispectrality in very general terms.
\begin{remk}
In our definition, we will use lower case for elements of algebras.
Since the algebras themselves are completely abstract, this is not a reflection on whether the elements themselves are scalar or matrix-valued.
Furthermore, the elements of each of the algebras are never assumed to be commutative.
However, the algebras are assumed to be associative with identity.
\end{remk}
\begin{defn}
We define an \vocab{operator algebra} to be an algebra $\mathcal A$ with a fixed subalgebra $\mathcal M(\mathcal A)$, referred to as the subalgebra of \vocab{multiplicative operators} of $\mathcal A$.
A \vocab{bispectral context} is a triple $(\mathcal A,\mathcal B, \mathcal H)$ where $\mathcal A$ and $\mathcal B$ are operator algebras and $\mathcal H$ is an $\mathcal A,\mathcal B$-bimodule.
A \vocab{bispectral triple} is a triple $(a,b,\psi)$ with $a\in\mathcal A,b\in\mathcal B$ nonconstant and $\psi\in\mathcal H$ satisfying the property that
$\psi$ has trivial left and right annihilator and
$$a\cdot\psi = \psi\cdot g,\ \text{and}\ \psi\cdot b = f\cdot\psi$$
for some $f\in\mathcal M(\mathcal A)$ and $g\in\mathcal M(\mathcal B)$.
In the case that $\psi$ forms part of a bispectral triple we call $\psi$ \vocab{bispectral}.
\end{defn}
In the classical case $\mathcal A$ and $\mathcal B$ are both algebras of differential operators with multiplicative operator subalgebras consisting of the functions on which the differential operators act (eg. polynomials, holomorphic functions, smooth function, etc).
However, as is clear from the definition, bispectrality is a far more general construction.   The generalization of specific relevance to this paper is the case when the multiplicative operator subalgebra $\mathcal M(\mathcal A)$ and $\mathcal M(\mathcal B)$ are themselves not commutative.

Given $\psi\in\mathcal H$ bispectral, we define certain natural subalgebras of $\mathcal A$ and $\mathcal B$.
\begin{defn}
Let $\psi\in\mathcal H$ be bispectral.
We define the \vocab{left and right Fourier algebras} $\fourl(\psi)$ and $\fourr(\psi)$ to be
$$\fourl(\psi) = \{a\in\mathcal A: \exists b\in\mathcal B,\ a\cdot\psi = \psi\cdot b\},\ \ \fourr(\psi) = \{b\in\mathcal B: \exists a\in\mathcal A,\ a\cdot\psi = \psi\cdot b\}.$$
We define the \vocab{left and right bispectral algebras} $\fourl(\psi)$ and $\fourr(\psi)$ to be
$$\bispl(\psi) = \{a\in\mathcal A: \exists g\in\mathcal M(\mathcal B),\ a\cdot\psi = \psi\cdot g\},\ \ \bispr(\psi) = \{b\in\mathcal B: \exists f\in\mathcal M(\mathcal A),\ f\cdot\psi = \psi\cdot b\}.$$
\end{defn}

Since the left and right annihilators of $\psi$ are both trivial, there is a natural isomorphism $b_\psi :\fourl(\psi)\rightarrow\fourr(\psi)$ defined by the identity
$$a\cdot\psi = \psi\cdot b_\psi(a).$$
\begin{defn}
Let $\psi\in\mathcal H$ be bispectral.
We call the natural isomorphism $b_\psi :\fourl(\psi)\rightarrow\fourr(\psi)$ defined above the \vocab{generalized Fourier map}.
\end{defn}

\begin{remk}
The name generalized Fourier map comes from our first example.  In that case $\mathcal A = \Omega(x),\mathcal B = \Omega(y)^{op}$ and $\psi = e^{xy}$.
Then $\fourl(\psi) = \Omega[x]$, $\fourr(\psi) = \Omega[y]^{op}$ and the generalized Fourier map is in fact the Fourier transform:
$$b_\psi: \sum_{i,j=0}^r a_{ij}x^i\partial_x^j\mapsto \sum_{i,j=0}^ra_{ij}y^j\partial_y^i.$$
Here and below, for an algebra $\mathcal{A}$, $\mathcal{A}^{op}$ denotes the one with opposite product.
\end{remk}

\subsection{Bispectral Darboux transformations}
In practice many of the important families of bispectral triples arise from simpler or more obvious examples through bispectral Darboux transformations.
\begin{defn}
\label{bisp-D}
Let $(\mathcal A,\mathcal B,\mathcal H)$ be a bispectral context, and let $\psi,\wt\psi\in\mathcal H$ be bispectral.
We say that $\wt\psi$ is a \vocab{bispectral Darboux transformation} of $\psi$ if there exist $u,\wt u\in \fourr(\psi)$ and units $p,\wt p\in \mathcal M(\mathcal A)$ and units $q,\wt q\in \mathcal M(\mathcal B)$ with 
\begin{equation}\label{bispectral darboux transformation}
\wt\psi = p^{-1}\cdot\psi\cdot uq^{-1}\ \ \text{and}\ \ \psi = \wt p^{-1}\cdot \wt\psi\cdot \wt q^{-1}\wt u.
\end{equation}
In the case that $\mathcal M(\mathcal A)$ or $\mathcal M(\mathcal B)$ is noncommutative, this is also called a \vocab{noncommutative bispectral Darboux transformation}
in \cite{geiger2017}.
\end{defn}
It follows from the definition that
$$\psi \cdot (u q^{-1}\wt q^{-1}\wt u) = p\wt p\cdot\psi,$$
$$(b_\psi^{-1}(\wt u)\wt p^{-1}p^{-1}b_\psi^{-1}(u))\cdot\psi = \psi\cdot \wt qq.$$
Furthermore, one works out that
$$\wt pp\cdot\wt\psi = \wt\psi\cdot(\wt q^{-1}\wt uuq^{-1}),$$
$$(p^{-1}b_\psi^{-1}(u)b_\psi^{-1}(\wt u)\wt p^{-1})\cdot\wt\psi = \wt\psi\cdot q\wt q,$$
see \cite[Theorem 2.1]{geiger2017}.
Thus both $\psi$ and $\wt\psi$ are bispectral.

The following symmetry property of bispectral Darboux transformations directly follows from the definition and its proof is left to the reader.
\begin{thm}
If $\wt\psi$ is a bispectral Darboux transformation of $\psi$ then $\psi$ is a bispectral Darboux transformation of $\wt\psi$.
\end{thm}

The bispectral nature of this type of Darboux transformations is established by the next theorem.

\begin{thm}[Geiger, Horozov and Yakimov \cite{geiger2017}]\label{lots of stuff}
Suppose that $\wt\psi$ is a bispectral Darboux transformation of $\psi$.
Then for all $a\in\fourl(\psi)$
$$\wt pap\cdot\wt\psi = \wt\psi\cdot\wt q^{-1}\wt ub_\psi(a) uq^{-1},$$
$$p^{-1}b_\psi^{-1}(u)ab_\psi^{-1}(\wt u)\wt p^{-1}\cdot\wt\psi = \wt\psi\cdot qb_\psi(a)\wt q.$$
\end{thm}

%In the continuous-continuous situation, the (matrix) bispectral Darboux transformations of the Airy and exponential functions were classified in 
%\cite[Theorems 1.2 and 1.3]{geiger2017}. These methods can be extended to obtain a classification of the (matrix) 
%bispectral Darboux transformations of the classical polynomials that appear in Theorem \ref{thm1.3}.  

\subsection{Degree-filtration preserving and exceptional differential operators}
Each differential operator $\mathfrak D\in \mathcal D(W)$ will have an entire $M_N(\bbc)$-module basis of $M_N(\bbc[x])$ as eigenfunctions.
As such for any polynomial $Q(x)\in M_N(\bbc[x])$ we know that the degree of $Q(x)\cdot\mathfrak D$ will be at most the degree of $Q(x)$.
We call such operators degree-filtration preserving.
\begin{defn}
A matrix-valued differential operator $\mathfrak D\in M_N(\bbc(x))$ is called \vocab{degree-filtration preserving} if for all polynomials $Q(x)\in M_N(\bbc[x])$ the function $Q(x)\cdot \mathfrak D$ is a polynomial whose degree is no larger than the degree of $Q(x)$.
\end{defn}
Degree-fitration preserving differential operators will necessarily have polynomial coefficients.
\begin{prop}
An operator $\mathfrak  D\in M_N(\bbc(x))$ is degree-filtration preserving if and only if
$$\mathfrak D = \sum_{i=0}^\ell \partial_x^i A_i(x)$$
with $A_i(x)\in M_N(\bbc[x])$ a polynomial of degree at most $i$ for all $i$.
\end{prop}
\begin{proof}
Suppose $\mathfrak D\in M_N(\Omega(x))$.
Then we may write
$$\mathfrak D = \sum_{i=0}^\ell \partial_x^i A_i(x)$$
for some matrix-valued rational functions $A_i(x)\in M_N(\bbc(x))$.
It is clear that if $A_i(x)\in M_N(\bbc[x])$ is a polynomial of degree at most $i$ for all $i$, then $\mathfrak D$ is degree-filtration preserving.

To prove the converse, suppose that $\mathfrak D$ is degree-filtration preserving and that for some $i$ the matrix-valued rational function $A_i(x)$ is either not a polynomial or has degree larger than $i$.
Let $j$ be the smallest non-negative integer with $A_j(x)$ not a polynomial of degree $\leq j$.
Then we calculate
$$\frac{1}{j!}x^j\cdot\mathfrak D = A_j(x) + \sum_{i=0}^{j-1} \frac{1}{(j-i)!}x^{j-i}A_i(x).$$
Since $\mathfrak D$ is degree-filtration preserving the left hand side of the above equality is a polynomial of degree at most $j$.
Furthermore the sum on the right hand side is a polynomial of degree at most $j$, so we conclude
$$ A_j(x) =  \sum_{i=0}^{j-1} \frac{1}{(j-i)!}x^{j-i}A_i(x) -\frac{1}{j!}x^j\cdot\mathfrak D$$
is a polynomial of degree at most $j$.
This is a contradiction and completes the proof.
\end{proof}

A Darboux conjugate $\wt{\mathfrak D}$ of $\mathfrak D$ will still have a great many polynomials which are eigenfunctions, but not necessarily a full basis for $M_N(\bbc[x])$.
For this reason $\wt{\mathfrak D}$ may have non-polynomial rational coefficients.
This leads us to consider the notion of an exceptional differential operator \cite[Definition 4.1]{garcia2016bochner}.
\begin{defn}
A differential operator $\mathfrak d\in\Omega(x)$ is called \vocab{exceptional} if it has polynomial eigenfunctions of all but finitely many degrees.  More precisely $\mathfrak d$ is exceptional if there exists a finite subset $\{m_1,\dots,m_k\}\subseteq \bbn$ satisfying the property that there exists a polynomial $p(n)$ of degree $n\in\bbn$ which is an eigenfunction of $\mathfrak d$ if and only if $n\notin \{m_1,\dots,m_k\}$.  The values $m_1,\dots,m_k$ are called the \vocab{exceptional degrees} of $\mathfrak d$.
\end{defn}

\begin{ex}
The classical differential operators
$$\partial_x^2 - \partial_x 2x,\ \ 
\partial_x^2x + \partial_x(b+1-x),\ \ 
\partial_x^2(1-x^2) + \partial_x(b-a + (b+a-2)x)$$
of Hermite, Laguerre and Jacobi respectively have polynomial eigenfunctions of every degree and are therefore exceptional operators.
\end{ex}

\begin{ex}
The differential operator
$$\mathfrak d= \partial_x^2 - \partial_x\left(2x + \frac{4x}{1+2x^2}\right)$$
is an exceptional differential operator with exceptional degrees $1$ and $2$.
\end{ex}

We can create new exceptional differential operators from old ones via Darboux conjugation.
\begin{defn}
Let $\mathfrak d,\wt{\mathfrak d}\in\Omega(x)$.
We say that $\wt{\mathfrak d}$ is a \vocab{Darboux conjugate} of $\mathfrak d$ if there exists a differential operator $\mathfrak h\in\Omega(x)$ satisfying $\mathfrak h\mathfrak d = \wt{\mathfrak d}\mathfrak h$.
\end{defn}

\begin{prop}
if $\wt{\mathfrak d}$ is a Darboux conjugate of $\mathfrak d$, then $\mathfrak d$ is a Darboux conjugate of $\wt{\mathfrak d}$.
\end{prop}
\begin{proof}
Assume that $\wt{\mathfrak d}$ is a Darboux conjugate of $\mathfrak d$.
Then there exists a differential operator $\mathfrak h\in\Omega(x)$ such that $\mathfrak h\mathfrak d = \wt{\mathfrak d}\mathfrak h$.
This implies that $\ker(\mathfrak h)\cdot\wt{\mathfrak d}\subseteq \ker(\mathfrak h)$.
The kernel is finite dimensional, so we may choose a polynomial $q(\wt{\mathfrak d})\in\bbc[\wt{\mathfrak d}]$ with $\ker(\mathfrak h)\cdot q(\wt{\mathfrak d}) = 0$.
This implies that there exists $\wt{\mathfrak h}$ satisfying $q(\wt{\mathfrak d}) = \mathfrak h\wt{\mathfrak h}$.
In particular $\wt{\mathfrak d}$ commutes with $\mathfrak h\wt{\mathfrak h}$ so that
$$\mathfrak h\mathfrak d\wt{\mathfrak h} = \wt{\mathfrak d}\mathfrak h\wt{\mathfrak h} = \mathfrak h\wt{\mathfrak h}\wt{\mathfrak d}.$$
Since $\Omega(x)$ is an integral domain, this implies $\mathfrak d\wt{\mathfrak h} = \wt{\mathfrak h}\wt{\mathfrak d}$.
Therefore $\mathfrak d$ is a Darboux conjugate of $\wt{\mathfrak d}$.
\end{proof}

\begin{prop}
Let $\mathfrak d$ be an exceptional differential operator and suppose that $\mathfrak d$ is a Darboux conjugate of $\wt{\mathfrak d}$.
Additionally assume $\mathfrak h\wt{\mathfrak d} = \mathfrak d\mathfrak h$ for $\mathfrak h\in \Omega[x]$.
Then $\wt{\mathfrak d}$ is also exceptional.
\end{prop}
\begin{proof}
For all but finitely many $n$, there exists a polynomial $p(x,n)$ of degree $n$ which is an eigenfunction of $\mathfrak d$.
This implies that $p(x,n)\cdot\mathfrak h$ is a polynomial eigenfunction of $\wt{\mathfrak d}$.
One may easily show that there exists an integer $m\in\bbz$ such that for all but finitely many $n$ the degree of $p(x,n)\cdot\mathfrak h$ is $n+m$.
Consequently $\wt{\mathfrak d}$ is exceptional.
\end{proof}

\begin{defn}
We call a differential operator $\mathfrak d\in\Omega(x)$ \vocab{degree-filtration preserving} if for all polynomials $p(x)\in\bbc[x]$ the function $p(x)\cdot\mathfrak d$ is a polynomial of degree at most the degree of $p(x)$.
\end{defn}

The most interesting fact about exceptional operators is that for low orders exceptional operators always arise as Darboux conjugates of degree preserving differential operators.
For order $0$ or $1$, one may in fact show that exceptional differential operators are necessarily degree-filtration preserving.
However for order $2$ this result requires significant calculation.
\begin{thm}[Garc\'{i}a-Ferrero, G\'{o}mez-Ullate and Milson \cite{garcia2016bochner}]\label{exceptional theorem}
If $\mathfrak d$ is an exceptional, second-order differential operator then $\mathfrak d$ is a Darboux conjugate to a degree-filtration preserving differential operator.
\end{thm}

\section{Bispectral Darboux transformations of weight matrices}
\subsection{$W$-adjoints of algebras of matrix differential and shift operators}
In this paper, we will be restricting our attention to a single specific bispectral context.
Specifically, we will consider the bispectral context $(M_N(\shift),M_N(\Omega[x])^{op},\mathcal P)$
where $\mathcal P$ is the set of all functions $P:\bbc\times\bbn\rightarrow M_N(\bbc)$ satisfying the property that for any fixed $n$, $P(x,n)$ is a matrix-valued rational function in $x$.  Equivalently, $P$ is the set of all semi-infinite sequences of matrix-valued rational functions.
The algebra $\shift$ is the collection of discrete operators in variable $n$.
The multiplicative subalgebra $\mathcal M(\shift)$ of $\shift$ consists of all functions $\bbn\rightarrow\bbc$ and $\shift$ itself is the set of operators $\text L$ of the form
$$\mathscr L = \sum_{j=0}^k a_j(n)\mathscr D^j + \sum_{j=1}^\ell b_j(n)(\mathscr D^*)^j$$
for $a_j,b_j:\bbn\rightarrow\bbc$ and $\mathscr D,\mathscr D^*$ the operators acting on sequences by
$$(\mathscr D\cdot a)(n) = a(n+1),\ \ \text{and}\ \ (\mathscr D^*\cdot a)(n) = \left\lbrace\begin{array}{cc}a(n-1), & n\neq 0\\ 0, & n=0\end{array}\right..$$
\begin{remk}
The operator $\mathscr D^*$ is adjoint to the operator $\mathscr D$ in the Hilbert space $\ell^2(\bbn)$, and this explains the notation.  Moreover $\mathscr{DD}^* = 1$ so $\mathscr D^*$ is a right inverse of $\mathscr D$.  However $\mathscr{D}^*\mathscr{D} = 1 - \delta_{0n}$, so $\mathscr D^*$ is not quite a left inverse of $\mathscr D$.  
\end{remk}
\begin{remk}
For us, sequences will always be indexed over $\bbn$.
With this in mind, we will always take the value of a sequence at a negative integer to be $0$, unless otherwise stated.
Furthermore, $n$ will be treated as an indeterminant.
\end{remk}

Now let $W(x)$ be a weight matrix and let $P(x,n)$ be the associated sequence of monic orthogonal polynomials.
Then the three-term recursion relation of $P(x,n)$ tells us that there exists $\mathscr L\in M_N(\shift)$ of the form
$$\mathscr L = \mathscr D + A(n) + B(n)\mathscr D^*$$
such that
$$\mathscr L\cdot P(x,n) = P(x,n)x.$$
Thus if $\mathcal D(W)$ contains a differential operator $\mathfrak D$ of positive order, then $P(x,n)$ will be bispectral with respect to this bispectral context.

As proved by Tirao and Gr\"unbaum \cite{grunbaum2007b}, the algebra $\mathcal D(W)$ is closed under the $W$-adjoint $\dag$.
Since the left and right bispectral algebras of $P$ are isomorphic, it makes sense that the involution $\dag$ of $\bispr(P)$ induces an involution on $\bispl(P)$.  Moreover, the involution $\dag$ is actually the restriction of an involution on a much larger algebra.  In fact, the right Fourier algebra $\fourr(P)$ is closed under $\dag$.  Thus $\fourl(P)$ is closed under the involution induced by the generalized Fourier map also.
To prove this, we first describe the induced involution.
First note that $M_N(\shift)$ has a natural $*$ operation extending Hermitian conjugation on matrices and sending $\mathscr D$ to $\mathscr D^*$ (and vice-versa).
Specifically it is defined by
\begin{equation}
\left(\sum_{j=0}^\ell A_j(n)\mathscr D^j + \sum_{j=1}^m (B_j(n)\mathscr D^*)^j\right)^* := \sum_{j=0}^\ell A_j(n-j)^*(\mathscr D^*)^j + \sum_{j=1}^m B_j(n+j)\mathscr D^j.
\end{equation}
With this in mind, we define the $W$-adjoint on $M_N(\shift)$.
\begin{defn}
The \vocab{$W$-adjoint of a shift operator} $\mathscr{M}\in M_N(\shift)$ is defined to be
$$\mathscr{M}^\dag := \|P(x,n)\|_W^2 \mathscr{M}^* \|P(x,n)\|_W^{-2},$$
where here we view $\|P(x,n)\|_W^2 = \langle P(x,n),P(x,n)\rangle_W$ as the sequence of Hermitian matrices in $M_N(\bbc)$ defined by the $W$-norms of the monic orthogonal polynomials of the weight matrix $W(x)$.
\end{defn}

The next lemma shows that this adjoint operation is well-behaved with respect to the inner product defined by the weight matrix $W(x)$.
\begin{lem}
Let $\mathscr{M}\in M_N(\shift)$.
We write $\mathscr{M}_n$ or $\mathscr{M}_m$ to emphasize that the shift operator is acting on the discrete variable $n$ or $m$, respectively.
Then consider the two polynomial sequences $P(x,n),P(x,m)$ in variables $n$ and $m$, respectively.
We have that
$$\langle \mathscr{M}_n\cdot P(x,n),P(x,m)\rangle_W = \langle P(x,n), \mathscr{M}_m^\dag\cdot P(x,m)\rangle_W.$$
\end{lem}
\begin{proof}
By linearity, it is enough to show that the above identity holds when $\mathscr M_n= A(n)\mathscr D^\ell$ or when $\mathscr M_n = A(n)(\mathscr D^*)^\ell$ for some sequence $A(n)$ and some integer $\ell\geq 0$.
We will prove the first case, leaving the second to the reader.
In this case $\mathscr M_n^* = A(n-\ell)^*(\mathscr D^*)^\ell$ so that
$$\mathscr M_n^\dag = \|P(x,n)\|_W^2A(n-\ell)^*\|P(x,n-\ell)\|^{-2}(\mathscr D^*)^\ell.$$
We calculate
\begin{align*}
\langle \mathscr M_n\cdot P(x,n),P(x,m)\rangle_W
  & = \langle A(n)\cdot P(x,n+\ell),P(x,m)\rangle_W\\
  & = A(n)\|P(m)\|_W^2\delta_{n+\ell,m}\\
  & = \|P(x,n)\|_W^2\|P(x,m-\ell)\|_W^{-2}A(m-\ell)\|P(x,m)\|_W^2\delta_{n,m-\ell}\\
  & = \langle P(x,n),\|P(x,m)\|_W^2 A(m-\ell)^*\|P(x,m-\ell)\|_W^{-2}\cdot P(x,m-\ell)\rangle_W\\
  & = \langle P(x,n),\mathscr M_m^\dag\cdot P(x,m)\rangle_W.
\end{align*}
This proves the lemma.
\end{proof}
\subsection{The Fourier algebras of orthogonal matrix polynomials}
We wish to describe the left and right Fourier algebras of $P$.
The next theorem describes the corresponding Fourier algebras in this case.
To prove it, we first require a lemma characterizing the differential operators contained in the algebra of left $M_N(\bbc)$-linear endomorphisms of $M_N(\bbc[x])$.

For every pair of elements $a, b$ of an algebra $\mathcal A$, we denote
$$
\Ad_a(b) := ab - ba \ \ \mbox{and} \ \ \Ad^{k+1}_a(b) := \Ad_a^k ( \Ad_a(b)).
$$
\begin{lem}
Suppose that $\Phi: M_N(\bbc[x])\rightarrow M_N(\bbc[x])$ is a left $M_N(\bbc)$-module endomorphism.
If $\Ad_{xI}^{\ell+1}(\Phi) = 0$ for some integer $\ell\geq 0$ then there exists a differential operator $\mathfrak D\in M_N(\Omega[x])$ of order at most $\ell$ satisfying $\Phi(P(x)) = P(x)\cdot\mathfrak D$ for all $P(x)\in M_N(\bbc[x])$.
\end{lem}
\begin{proof}
First recall that any left $M_N(\bbc[x])$-module endomorphism of $M_N(\bbc[x])$ is of the form $P(x)\mapsto P(x)A(x)$ for some $A(x)\in M_N(\bbc[x])$.
With this in mind, we proceed by induction on $\ell$.
If $\ell = 0$, then $\Ad_{xI}(\Phi) = 0$ so that $\Phi(xP(x)) = x\Phi(P(x))$ for all $P(x)\in M_N(\bbc[x])$.
Since $\Phi(x)$ is a left $M_N(\bbc)$-module endomorphism of $M_N(\bbc[x])$ it follows that $\Phi$ is a left $M_N(\bbc[x])$-module endomorphism.
Thus by the fact recalled at the start our lemma is true when $\ell=0$.

As an inductive assumption, suppose that the statement of the lemma is true for all $\ell \leq m$.
Now suppose that $\Ad_{xI}^{m+1}(\Phi) = 0$.
Then $\Ad_{xI}(\Phi)$ is a left $M_N(\bbc)$-module endomorphism of $M_N(\bbc[x])$ with $\Ad_{xI}^m(\Ad_{xI}(\Phi)) = 0$.
Thus by our inductive assumption there exists a differential operator $\mathfrak D\in M_N(\bbc[x])$ of order at most $m$ satisfying
$$\Ad_{xI}(\Phi)(P(x)) = P(x)\cdot\mathfrak D$$
for all $P(x)\in M_N(\bbc[x])$.
We write $\mathfrak D = \sum_{j=0}^m \partial_x^j A_j(x)$ and define a new left $M_N(\bbc)$-module endomorphism $\Psi$ of $M_N(\bbc[x])$ by
$$\Psi(P(x)) = \Phi(P(x))-\left[P(x)\cdot \sum_{j=0}^m \partial_x^{j+1}\frac{1}{j+1} A_j(x)\right].$$
We calculate
$$\Ad_{xI}(\Psi)(P(x)) = \Psi(xP(x))-x\Psi(P(x)) = \Ad_{xI}(\Phi)(P(x))-\left[P(x)\cdot\sum_{j=0}^m\partial_x^jA_j(x)\right] = 0.$$
Therefore there exists a matrix $A(x)\in M_N(\bbc[x])$ with $\Psi(P(x)) = P(x)A(x)$ for all $P(x)\in M_N(\bbc[x])$.
Hence for all $P(x)$ we have
$$\Phi(P(x)) = P(x)\cdot\left[\sum_{j=0}^m \partial_x^{j+1}\frac{1}{j+1}A_j(x) + A(x)\right].$$
Hence by induction our lemma is true for all $\ell$.
\end{proof}

\begin{thm}\label{fourier algebra theorem}
Let $W(x)$ be a weight matrix and let $P(x,n)$ be the associated sequence of monic orthogonal polynomials and let $\mathscr L\in M_N(\shift)$ with $\mathscr L\cdot P(x,n) = P(x,n)x$.
Then the Fourier algebras of $P(x,n)$ are given by
\begin{align*}
\fourl(P) &= \{\mathscr{M} \in M_N(\shift): \Ad_{\mathscr L}^{k+1}(\mathscr{M}) = 0\ \text{for some $k\geq0$}\},
\\
\fourr(P) &= \{\mathfrak D\in M_N(\Omega[x])^{op}: \text{$\mathfrak D$ is $W$-adjointable and $\mathfrak D^\dag\in M_N(\Omega[x])$}\}.
\end{align*}
\end{thm}
\begin{proof}
We will first prove our formula for $\fourl(P)$.
If $\mathscr M\in\fourl(P)$, then $\mathscr M\cdot P(x,n) = P(x,n)\cdot\mathfrak D$.
If $\mathfrak D$ has order $\ell$, then $\Ad_{xI}^{\ell+1}(\mathfrak D) = 0$.
Applying the generalized Fourier map, this implies that $\Ad_{\mathscr L}^{\ell+1}(\mathscr M) = 0$.
Thus
$$\fourl(P) \subseteq \{\mathscr M\in M_N(\shift): \Ad_{\mathscr L}^{k+1}(M) = 0\ \text{for some $k\geq0$}\}.$$

To prove the opposite containment, suppose that $\mathscr M\in M_N(\shift)$ with $\Ad_{\mathscr L}^{\ell+1}(M) = 0$ for some integer $\ell$.
Consider the left $M_N(\bbc)$-module endomorphism $\Phi$ of $M_N(\bbc[x])$ induced by
$$\Phi: P(x,n)\mapsto \mathscr M\cdot P(x,n).$$
For any $Q(x)\in M_N(\bbc[x])$, we write $Q(x)\cdot\Phi$ to mean $\Phi(Q)(x)$.  Then for all $n$ we see
$$P(x,n)\cdot \Ad_{xI}^{\ell+1}(\Phi) = \Ad_{\mathscr L}^{\ell+1}(\mathscr{M})\cdot P(x,n) = 0.$$
Consequently $\Ad_{xI}^{\ell+1}(\Phi) = 0$ and by the previous lemma we know that
$$\mathscr M\cdot P(x,n) = \Phi(P(x,n)) = P(x,n)\cdot\mathfrak D$$
for some differential operator $\mathfrak D\in M_N(\Omega[x])$ for all $n$.
In particular $\mathscr M\in\fourl(P)$.  This proves
$$\fourl(P) = \{\mathscr M\in M_N(\shift): \Ad_{\mathscr L}^{k+1}(\mathscr{M}) = 0\ \text{for some $k\geq0$}\}.$$

Suppose that $\mathfrak D\in M_N(\Omega[x])$ and $\mathfrak D$ is $W$-adjointable with $\mathfrak D^\dag\in M_N(\Omega[x])$.
Then we may write
$$\mathfrak D = \sum_{i=0}^\ell \partial_x^i A_i(x),\ \ \text{and}\ \ \mathfrak D^\dag = \sum_{i=0}^\ell \partial_x^i B_i(x)$$
for some polynomials $A_i(x),B_i(x)\in M_N(\bbc[x])$.  Let $m$ be the maximum degree of all of these polynomials.  Then we may write
$$P(x,n)\cdot\mathfrak D = \sum_{j=0}^{n+m} C(n,j)P(x,j)$$
for some matrices $C(n,j)\in M_N(\bbc)$, defined by
$$C(n,j)  = \langle P(x,n)\cdot\mathfrak D,P(x,j)\rangle_W \|P(x,j)\|_W^{-2}.$$
Then using the $W$-adjointability, we see that
$$C(n,j)  = \langle P(x,n),P(x,j)\cdot\mathfrak D^\dag\rangle_W \|P(x,j)\|_W^{-2}$$
and therefore $C(n,j)$ is zero if $n-j>m$.  Thus
$$P(x,n)\cdot\mathfrak D = \sum_{j=n-m}^{n+m} C(n,j)P(x,j) = \mathscr M\cdot P(x,n)$$
for
$$\mathscr M = \sum_{j=0}^{m} C(n,n+j)\mathscr D^j + \sum_{j=1}^m C(n,n-j)(\mathscr D^*)^j.$$
Therefore $\mathfrak D\in\fourr(P)$.
Moreover, since $\Ad_x^{\ell+1}(\mathfrak D) = 0$ we know that $\Ad_{\mathscr L}^{\ell+1}(\mathscr M) = 0$.
This proves
$$\fourr(P) \supseteq \{\mathfrak D\in M_N(\Omega[x])^{op}: \text{$\mathfrak D$ is $W$-adjointable and $\mathfrak D^\dag\in M_N(\Omega[x])$}\}.$$
Next, suppose instead that $\mathfrak D\in\fourr(P)$.
Then there exists $\mathscr M\in M_N(\shift)$ with $\mathscr M\cdot P(x,n) = P(x,n)\cdot \mathfrak D$.
Then $\Ad_{\mathscr L}^{\ell+1}(\mathscr M) = 0$ for $\ell$ the order of $\mathfrak D$.
Since $xI$ is $\dag$-symmetric, so too is $\mathscr L$.
It follows that $\Ad_{\mathscr L}^{\ell+1}(\mathscr M^\dag) = 0$ and therefore $\mathscr M^\dag \cdot P(x,n) = P(x,n)\cdot\wt{\mathfrak D}$ for some $\wt{\mathfrak D}\in\fourr(P)$.
Note that by the previous lemma for all $m\in\bbz$
\begin{align*}
\langle P(x,n)\cdot\mathfrak D,P(x,m)\rangle_W
  & = \langle \mathscr M_n\cdot P(x,n),P(x,m)\rangle_W\\
  & = \langle P(x,n),\mathscr M_m^\dag\cdot P(x,m)\rangle_W\\
  & = \langle P(x,n), P(x,m)\cdot\wt{\mathfrak D}\rangle_W.
\end{align*}
This implies that $\langle P\cdot\mathfrak D,Q\rangle_W = \langle P,Q\cdot\wt{\mathfrak D}\rangle_W$ for all polynomials $P,Q\in M_N(\bbc[x])$.
Hence $\mathfrak D$ is $W$-adjointable and $\wt{\mathfrak D} = \mathfrak D^\dag$.
Thus $\mathfrak D^\dag$ is $W$-symmetric with $\mathfrak D^\dag\in M_N(\Omega[x])$.
This proves
$$\fourr(P) = \{\mathfrak D\in M_N(\Omega[x])^{op}: \text{$\mathfrak D$ is $W$-adjointable and $\mathfrak D^\dag\in M_N(\Omega[x])$}\}.$$
\end{proof}
As a corollary of this, we see that the left and right Fourier algebras are closed under the adjoint operation.
In this way, the previous theorem provides a more conceptual proof of the fact $\mathcal D(W)$ is closed under $\dag$ than the proof found in \cite{grunbaum2007b}, 
see Theorem \ref{GrTi} below.
\begin{cor}\label{fourier algebra corollary}
Let $W(x)$ be a weight matrix and $P(x,n)$ be the associated sequence of monic orthogonal polynomials.
The left and right Fourier algebras $\fourl(P)$ and $\fourr(P)$ are closed under $\dag$.
Furthermore for all $\mathscr M\in\fourl(P)$ we have $b_P(\mathscr M^\dag) = b_P(\mathscr M)^\dag$.
\end{cor}
\begin{proof}
It is clear from the statement of the previous theorem that $\fourr(P)$ is closed under $\dag$.
We showed in the proof of the previous theorem that if $\mathfrak b_P(\mathscr M) = \mathfrak D$, then $b_P(\mathscr M^\dag) = \mathfrak D^\dag$.
Therefore the identity in the statement of the corollary is true.
Since $b_P$ is an isomorphism, this implies $\fourl(P)$ is also closed under $\dag$.
\end{proof}

\begin{defn}
Let $W(x)$ and $\wt W(x)$ be weight matrices and let $P(x,n)$ and $\wt P(x,n)$ be their associated sequences of monic orthogonal polynomials.
We say that $\wt P(x,n)$ is a bispectral Darboux transformation of $P(x,n)$ if there exist differential operators $\mathfrak T,\wt{\mathfrak T}\in\fourr(P)$, polynomials $F(x),\wt F(x)$ and sequences of matrices $C(n),\wt C(n)$ which are nonsingular for almost all $n$ and satisfy
$$C(n)\wt P(x,n) = P(x,n)\cdot \mathfrak TF(x)\ \ \text{and}\ \ \wt C(n)P(x,n) = \wt P(x,n)\cdot \wt F(x)\wt{\mathfrak T}.$$
We say that $\wt W(x)$ is a \vocab{bispectral Darboux transformation} of $W(x)$ if $\wt P(x,n)$ is a bispectral Darboux transformation of $P(x,n)$.
\end{defn}
\begin{remk}
We do not require the $C(n),\wt C(n)$'s to be nonsingular for all $n$, because doing so would eliminate many important bispectral Darboux transformations, including many trivial ones.
In particular the vanishing of $C(n)$ or $\wt C(n)$ corresponds to ``poles" of the eigenvalues, which occur naturally.
For example, we should expect any $\mathfrak T,\wt{T}\in\bispr(P)$ which are not zero divisors to define a bispectral Darboux transformation of $P$ to itself.
However, if we required the $C(n),\wt C(n)$ to not have poles then these trivial transformations would have to be thrown out.
This is a feature of the ``discrete" bispectral case not featured in the continuous case, since in the latter we deal with meromorphic eigenvalues which can already have poles.
\end{remk}
The next example shows that one can use bispectral Darboux transformations to decrease the minimal nonzero degree of an element of $\mathcal D(W)$ by
artificially increasing the size of a matrix weight $W(x)$.
\begin{ex} 
\label{double-size}
Bispectral Darboux transformations can be used to artificially create weight matrices $W(x)$ for which the algebra $\mathcal D(w)$ contains 
differential operators of lower order at the expense of enlarging the size of the weight by the following procedure.

Let $W(x)$ and $\wt W(x)$ be a pair of weight matrices of size $N \times N$ which are bispectral Darboux transformations from each other. Denote by 
$P(x,n)$ and $\wt P(x,n)$ their associated sequences of monic orthogonal polynomials. Thus 
$$
\wt{P} = p^{-1}\cdot P \cdot {\mathfrak{d}} q^{-1}\ \ \text{and}\ \ P = \wt p^{-1}\cdot \wt{P} \cdot \wt q^{-1}\wt{\mathfrak{d}}
$$
for some ${\mathfrak{d}}, \wt{\mathfrak{d}} \in \mathcal{F}_R(P)$ and some functions $q, \wt{q} \in \mathcal{F}_R(P)$, $p, \wt{p} \in \mathcal{F}_L(P)$. 

The two sequences $P(x,n)$ and $\wt P(x,n)$ are eigenfunctions of differential operators of orders ${\mathrm{ord}} P + {\mathrm{ord}} \wt{P}$:
$$P \cdot ({\mathfrak{d}} q^{-1}\wt q^{-1} \wt{\mathfrak{d}}) = p\wt p\cdot P
\ \ \mbox{and} \ \
\wt pp\cdot\wt{P} = \wt{P} \cdot(\wt q^{-1}\wt{\mathfrak{d}} {\mathfrak{d}} q^{-1}).
$$

At the same time, the sequence of orthogonal polynomials for the block diagonal weight $\diag (W(x), \wt{W}(x))$ is $\diag (P(x,n), \wt P(x,n))$. It is an eigenfunction 
of a differential operator of order $\max\{ {\mathrm{ord}} P, {\mathrm{ord}} \wt{P} \}$,
$$
\diag(\wt p, p) \cdot \diag (P, \wt P) = \diag (P, \wt P) \cdot \mxx{0}{ {\mathfrak{d}} q^{-1} }{\wt{q}^{-1}\wt{\mathfrak{d}}}{0}.
$$
Note that the leading coefficient of the operator in the right hand side does not have the positivity property from Theorem \ref{thm1.3}.
\end{ex}

\section{The algebraic structure of $\mathcal{D}(W)$}
In this section, we describe the algebraic structure of the algebra $\mathcal{D}(W)$ in very specific terms.
\subsection{The $*$-involution of $\mathcal D(W)$}
The structural results featured in this section are dependent on the existence of a proper $*$-involution on $\mathcal{D}(W)$.
Recall that a (complex) \vocab{$*$-algebra} is an algebra $\mathcal A$ over $\mathbb{C}$ with an involutive skew-linear anti-automorphism $\mathfrak{D} \mapsto \mathfrak{D}^\dag$.
One of the first general results regarding the algebra $\mathfrak{D}(W)$ is that it is a $*$-algebra under the $W$-adjoint $\dag$.
\begin{thm}[Gr\"{u}nbaum and Tirao \cite{grunbaum2007b}]
\label{GrTi}
Let $W(x)$ be a weight matrix, and let $\mathfrak{D}\in \mathcal{D}(W)$.  Then a $W$-adjoint of $\mathfrak{D}$ exists and is in $\mathcal{D}(W)$.  The operator $\mathfrak{D}\mapsto\mathfrak{D}^\dag$ is an involution on $\mathcal{D}(W)$ giving $\mathcal{D}(W)$ the structure of a $*$-algebra.
\end{thm}
This is also a consequence of Corollary \ref{fourier algebra corollary} in the previous section.
As a result, we have the following corollary proved in \cite{grunbaum2007b}.
\begin{cor}
Let $W(x)$ be a weight matrix.  Then $\mathcal{D}(W)$ contains a differential operator of order $m$ if and only if $\mathcal{D}(W)$ contains a $W$-symmetric differential operator of order $m$.
\end{cor}

\begin{defn}
The involution $\mathfrak{D} \mapsto \mathfrak{D}^\dag$ of a $*$-algebra $\mathcal A$ is called \vocab{positive} if for all $\mathfrak D_1, \ldots, \mathfrak D_n \in \mathcal A$, 
\begin{equation}\label{nondegenerate star}
\mathfrak D_1\mathfrak D_1^\dag  + \cdots + \mathfrak D_n\mathfrak D_n^\dag = 0 \ \text{if and only if}\ \mathfrak D_1 = \cdots = \mathfrak D_n = 0.
\end{equation}
The involution is called \vocab{proper} if the above property holds for $n=1$, see \cite[Sect. 2]{berberian1972}.
\end{defn}
\begin{thm}[Casper \cite{casper2017bispectral}]
Let $W(x)$ be a weight matrix.  The involution of the $*$-algebra $\mathcal{D}(W)$ is positive. As a consequence $\mathcal{D}(W)$ is a semiprime 
PI-algebra, and thus its center $\mathcal Z(W)$ is reduced.
\end{thm}
\begin{proof}
The generalized Fourier map defined in the next section embeds $\mathcal{D}(W)$ into the matrix ring $M_N(\bbc[n])$ of polynomials in variable $n$.  Hence $\mathcal{D}(W)$ is a PI-algebra.

To prove that the involution of the $*$-algebra $\mathcal{D}(W)$ is positive, 
consider $\mathfrak D_1, \ldots, \mathfrak D_n \in\mathcal {D}(W)$ such that
$\mathfrak D_1\mathfrak D_1^\dag  + \cdots + \mathfrak D_n\mathfrak D_n^\dag = 0$. Then for all polynomials $P(x)$ we have that
$$0 = \langle P(x)\cdot (\mathfrak D_1\mathfrak D_1^\dag  + \cdots + \mathfrak D_n\mathfrak D_n^\dag), P(x)\rangle_W = \sum_k
\langle P(x)\cdot\mathfrak D_k,P(x)\cdot \mathfrak D_k\rangle_W = \sum_k \|P(x)\cdot\mathfrak D_k \|_W^2.$$
Therefore $P(x)\cdot\mathfrak D_k = 0$ for all polynomials $P(x)$ and $1 \leq k \leq n$ which implies $\mathfrak D_k = 0$.  
%Thus $\mathfrak D\mathfrak D^\dag = 0$ if and only if $\mathfrak D = 0$.

Every $*$-algebra $\mathcal{A}$ with a proper involution is semiprime. For this one needs to show that for any 
$\mathfrak D \in \mathcal{A} $, the condition $\mathfrak D \mathcal{A} \mathfrak D = 0$ implies $\mathfrak D = 0$. For consistency, denote the $*$-operation of $\mathcal{A}$ by $\dag$. 
Suppose $\mathfrak D \mathcal{A} \mathfrak D = 0$.  Then $\mathfrak D\mathfrak D^\dag\mathfrak D = 0$ and therefore
$$0 = \mathfrak D\mathfrak D^\dag\mathfrak D\mathfrak D^\dag = (\mathfrak D\mathfrak D^\dag)(\mathfrak D\mathfrak D^\dag)^\dag.$$
Applying (\ref{nondegenerate star}) gives $\mathfrak D\mathfrak D^\dag = 0$ and, applying (\ref{nondegenerate star}) one more time, gives $\mathfrak D = 0$.  

The center of every semiprime algebra $\mathcal{A}$ is reduced. If $\mathfrak D \in \mathcal Z(\mathcal{A})$ and $\mathfrak D^2 =0$, 
then $\mathfrak D \mathcal{A} \mathfrak D =0$, hence $\mathfrak D =0$. 
\end{proof}

\subsection{The center of $\mathcal{D}(W)$}
In this section we will use the involution of $\mathcal D(W)$ to prove that the center $\mathcal{Z}(W)$ of $\mathcal{D}(W)$ 
is a Noetherian algebra of Krull dimension at most $1$ and that the minimal prime ideals of $\mathcal{Z}(W)$ are preserved by the $W$-adjoint of $\mathcal{D}(W)$.
There are simple examples of commutative subalgebras of $M_2(\bbc[n])$ which are not finitely generated, such as $\bbc I\oplus \mxx{0}{\bbc[n]}{0}{0}$, so the presence of the involution is crucial to our arguments.

% FIXME: remember to mention viewing n as a continuous variable
\begin{lem}
Let $\mathcal R_1,\dots,\mathcal R_\ell$ be any collection of Dedekind domains whose fraction fields are finite extensions of $\bbc(t)$.
Then any subalgebra of $\mathcal R_1\times\dots\times\mathcal R_\ell$ is finitely generated.
\end{lem}
\begin{proof}
Note that each $\mathcal R_i$ is isomorphic to the ring of functions of $X_i\diff\{p\}$ for some smooth, projective algebraic curve $X_i$ over $\bbc$ and some closed point $p_i\in X_i$ called the point at infinity.
The local ring of $X_i$ at $p_i$ is isomorphic to a discrete valuation ring $S_i$ with fraction field $\mathcal L_i$ and residue field $\bbc$.
The completion of this ring at its maximal ideal is isomorphic to $\bbc[[t]]$, and thus we have an injection $\mathcal L_i\subseteq\bbc((t))$ into the ring of Laurent series with coefficients in $\bbc$.
Restriction of functions on $X_i\diff\{p_i\}$ to $(X_i\diff\{p_i\})\cap\spec(S_i)$ defines an injection $\mathcal R_i\rightarrow\mathcal L_i$ and thereby an embedding 
$\phi_i:\mathcal R_i\rightarrow \bbc((t))$.
Note that for any nonzero $r\in \mathcal R_i$ the image $\phi_i(r)$ has nonpositive degree.

Let $\mathcal A$ be a subalgebra of $\mathcal R_1\times \dots \times \mathcal R_\ell$.
For any nonzero $A = (r_1,\dots,r_\ell)\in\mathcal A$, we define
$$\deg(A) = \max_{r_j\neq 0}\deg(\phi_j(r_j))\ \ \text{and}\ \ \supp(A) = \{i: \deg(r_i) = \deg(A)\}.$$
Further, we define
$$\mathcal A_m = \{A\in \mathcal A: A=0\ \text{or}\ \deg(A) \geq -m\}.$$
We define $\chi: \mathcal R_1\times\dots\times\mathcal R_\ell\rightarrow\bbc^{\oplus \ell}$ by setting $\chi(r_1,\dots,r_\ell)$ to be the vector whose $i$'th entry is the leading coefficient of $r_i$ if $\nu_i(r_i) = \nu(r_1,\dots,r_\ell)$ and $0$ otherwise (with $\chi(0) = 0$).
Note that $\chi$ induces a $\bbc$-linear injection $\mathcal A_{m+1}/\mathcal A_m\rightarrow\bbc^{\oplus\ell}$ so that in particular $\dim(\mathcal A_{m+1}/\mathcal A_m)\leq \ell$.

Note that $\supp(A^j) = \supp(A)$.
Now consider any two elements $A,B\in\mathcal A$ with $\nu(A) = m$ and $\nu(B) = n$.
Then for some nonzero $\lambda\in\bbc$ we have
$$\supp(A^n + \lambda B^m) = \supp(A)\cup \supp(B).$$ 
Thus there exists $C\in\mathcal A$ such that $\supp(A)\subseteq\supp(C)$ for all $A\in \mathcal A$. Denote $k = \nu(C)$.  
Multiplication by  $C$ induces an injection $(\mathcal A_{m+1}/\mathcal A_m) \rightarrow (\mathcal A_{m+k+1}/\mathcal A_{m+k})$ for all $m$.
Since the dimension of $(\mathcal A_{m+1}/\mathcal A_m)$ is bounded by $\ell$, there must exists $n>0$ such that for all $m>n$ the above map is an isomorphism.
It follows that $\mathcal A$ is generated by a $\bbc$-linear basis for $\mathcal A_{n+k}$, keeping in mind that $C \in \mathcal A_{n+k}$.
In particular $\mathcal A$ is finitely generated.
\end{proof}

\begin{lem}
Let $\mathcal A\subseteq M_N(\bbc[t])$ be a subalgebra of the matrix algebra with coefficients in the polynomial ring $\bbc[t]$ in an indeterminant $t$.
If all of the elements of $\mathcal A$ are simultaneously diagonalizable over the algebraic closure $\ol{\bbc(t)}$ of $\bbc(t)$, then $\mathcal A$ is finitely generated as an algebra over $\bbc$.
\end{lem}
\begin{proof}
Let $\mathcal F = \bbc(t)$ and suppose that $\mathcal A$ is simultaneously diagonalizable over the algebraic closure $\ol{\mathcal F}$.
Let $\mathcal B$ be the $\mathcal F$-subalgebra of $M_N(\mathcal F)$ spanned by $\mathcal A$.
Note in particular that all the elements of $\mathcal B$ are simultaneously diagonalizable.
As a consequence, the nilradical of $\mathcal B$ is $0$ and since $\mathcal B$ is finite dimensional over $\mathcal F$ we conclude that $\mathcal B$ is semisimple.
Therefore by Artin-Wedderburn, $\mathcal B$ is isomorphic to a direct product of matrix algebras over skew-fields over $\mathcal F$.
However, since $\mathcal B$ is commutative this in fact implies that $\mathcal B$ is isomorphic to a direct product of field extensions
$$\mathcal B \cong \mathcal L_1\times\dots\times\mathcal L_\ell$$
where each $\mathcal L_i$ is a finite extension of the field $\mathcal F$.

For each $i$, let $\mathcal R_i$ denote the integral closure of $\bbc[t]$ in $\mathcal L_i$. In particular it satisfies the assumptions 
of the previous lemma.
The algebra $\mathcal A$ is contained in $M_N(\bbc[t])$, so each element of $\mathcal A$ is integral over $\bbc[t]$.
Hence the image of $\mathcal A$ in $\mathcal L_1\times\dots\times\mathcal L_\ell$ is contained in $\mathcal R_1\times\dots\times\mathcal R_\ell$.
By the previous lemma, $\mathcal A$ must be finitely generated.
\end{proof}

\begin{thm}\label{center is noetherian}
The center $\mathcal Z(W)$ of $\mathcal D(W)$ is a Noetherian algebra of Krull dimension at most $1$.
\end{thm}
\begin{proof}
First recall that the generalized Fourier map embeds $\mathcal D(W)$ as a subalgebra $\mathcal A$ of $M_N(\bbc[n])$.
This latter ring has GK dimension $1$, so the GK dimension of $\mathcal A$ must be at most $1$.
Therefore the GK dimension of the center $\mathcal Z$ of $\mathcal A$ is at most $1$.
For a commutative ring the GK dimension and the Krull dimension are the same, so the Krull dimension of $\mathcal Z$ is at most $1$.

Let $\mathcal A\subseteq M_N(\bbc[n])$ denote the preimage of $\mathcal D(W)$ under the generalized Fourier map $b_P$ and let $\mathcal Z$ denote the center of $\mathcal A$.
The involution $\dag$ on $\mathcal A$ is an anti-automorphism of $\mathcal A$ and therefore must restrict to an automorphism of $\mathcal Z$.
Furthermore for any $W$-symmetric $A(n)\in \mathcal A$ we know that $\|P(x,n)\|_W^{-1}A(n)\|P(x,n)\|_W$ is Hermitian and therefore diagonalizable.
Therefore $A(n)$ evaluates to a diagonalizable matrix for every integer value of $n$.
Since $A(n)$ is polynomial in $n$, it follows that $A(n)$ is diagonalizable for all values of $n$.
The algebra $\mathcal Z$ is commutative and spanned over $\bbc$ by its $W$-symmetric elements, so therefore all of the elements of $\mathcal Z$ are simultaneously diagonalizable.
By the previous lemma this implies that $\mathcal Z$ is finitely generated as an algebra over $\bbc$.
The statement of the theorem follows immediately.
\end{proof}

\begin{prop}
The minimal prime ideals $\mathcal P_1(W),\dots,\mathcal P_r(W)$ of $\mathcal Z(W)$ are fixed by the adjoint operation $\dag$.
\end{prop}
\begin{proof} For brevity set $\mathcal P_i := \mathcal P_i(W)$ and $\mathcal Z := \mathcal Z(W)$.
The involution $\dag$ defines an anti-isomorphism of $\mathcal D(W)$ and therefore restricts to an isomorphism of $\mathcal Z$.
Consequently the image $\mathcal P_i^\dag$ of the minimal prime ideal $\mathcal P_i$ must be a minimal prime ideal also.
If $\mathcal Z$ has only one minimal prime, then the result of the proposition is trivial.
Therefore we assume otherwise, so that in particular each of the minimal prime is not the zero ideal.

For each $i$, consider the annihilator ideal of $\mathcal P_i$ defined by
$$\text{Ann}(\mathcal P_i) = \{\mathfrak D\in \mathcal Z(W): \mathfrak D\mathcal P_i =  0 \}.$$
Since $\mathcal Z(W)$ is Noetherian, the ideal $\mathcal P_i$ is finitely generated by some elements $\mathfrak D_1,\dots,\mathfrak D_m\in\mathcal Z(W)$.
For each $i$ there exists $\mathfrak D_i'\in \mathcal Z(W)\diff\mathcal P_i$ with $\mathfrak D_i\mathfrak D_i' = 0$.
The set $\mathcal Z(W)\diff\mathcal P_i$ is multiplicatively closed and does not contain zero, so $\mathfrak D' = \mathfrak D_1'\dots\mathfrak D_m'$ is nonzero.
Since $\mathfrak D'$ annihilates the generators of $\mathcal P_i$ this proves $\text{Ann}(\mathcal P_i)$ is nonzero.

For all $j\neq i$, by definition $\text{Ann}(\mathcal P_i)\mathcal P_i = 0\subseteq \mathcal P_j$.  Since $\mathcal P_j$ is prime either $\text{Ann}(\mathcal P_i)\subseteq\mathcal P_j$ or $\mathcal P_i\subseteq \mathcal P_j$.
However since $\mathcal P_j$ is a minimal prime, it must be the former.
Hence $\text{Ann}(\mathcal P_i)\subseteq \bigcap_{j\neq i}\mathcal P_j$.

Choose a nonzero $\mathfrak D\in\text{Ann}(\mathcal P_i)$.
Since $\mathfrak D\mathfrak D^\dag \neq 0$, this implies that $\mathfrak D^\dag\notin\mathcal P_i$ and therefore $\mathfrak D\notin\mathcal P_i^\dag$.
By the previous paragraph $\mathfrak D\in \mathcal P_j$ for all $j\neq i$, so therefore $\mathcal P_i^\dag\neq \mathcal P_j$ for all $j\neq i$.
Thus $\mathcal P_i^\dag = \mathcal P_i$.
This completes the proof.
\end{proof}

\subsection{$\mathcal D(W)$ is finite over its center}
The results of the previous section tell us that $\mathcal Z(W)$ is reduced and Noetherian with Krull dimension $1$.
In this section we will prove that the algebra $\mathcal D(W)$ is affine (i.e. finitely generated as an algebra over $\bbc$), from which we conclude that it is Noetherian and finitely generated over its center.

The spectrum of the center will consist of a disjoint union of affine curves and points.
As in the previous section, let $\mathcal P_1,\dots,\mathcal P_r$ denote the minimal primes of $\mathcal Z(W)$.
Let $I\subseteq \{1,\dots,r\}$ be the set of indices of minimal primes corresponding to discrete points of $\spec(\mathcal Z(W))$, i.e.
$$I = \{i: \mathcal Z(W)/\mathcal P_i\cong\bbc\}.$$
For each $i\in I$, there exists a $W$-symmetric idempotent $\mathfrak E_i\in\mathcal Z(W)$ which maps to $1$ in $\mathcal Z(W)/\mathcal P_j$ if $j=i$ and to $0$ otherwise.
Using the system of orthogonal idempotents $\{\mathfrak E_i: i\in I\}$ the rings $\mathcal D(W)$ and $\mathcal Z(W)$ decompose as a direct sums
$$\mathcal D(W) = \mathcal D_{disc}(W) \oplus \mathcal D_{cont}(W);\ \ \mathcal D_{disc}(W) = \bigoplus_{i\in I}\mathfrak E_i\mathcal Z(W)\mathfrak E_i,\ \mathcal D_{cont}(W) = \mathfrak E\mathcal D(W)\mathfrak E$$
$$\mathcal Z(W) = \mathcal Z_{disc}(W) \oplus \mathcal Z_{cont}(W);\ \ \mathcal Z_{disc}(W) = \bigoplus_{i\in I}\mathfrak E_i\mathcal Z(W)\mathfrak E_i,\ \mathcal Z_{cont}(W) = \mathfrak E\mathcal Z(W)\mathfrak E$$
for $\mathfrak E = I-\sum_{i\in I}\mathfrak E_i$.
Each of the above summands is an algebra and the above identifications are algebra isomorphisms.
Since $\mathcal D(W)$ is a semiprime PI-algebra, both $\mathcal D_{disc}(W)$ and $\mathcal D_{cont}(W)$ must be semiprime PI-algebras.

The next lemma captures the crucial algebraic consequence of our involution.
\begin{lem}\label{super lemma}
Consider a weight matrix $W(x)$ with sequence of monic orthogonal matrix polynomials $P(x,n)$.
Let $\mathfrak D\in\mathcal D(W)$ and let $A(n)\in M_N(\bbc[n])$ be the preimage of $\mathfrak D$ under the generalized Fourier map $b_P$.
If the eigenvalues of $A(n)A(n)^\dag$ are all constant, then both $A(n)$ and $\mathfrak D$ are constant.
\end{lem}
\begin{proof}
Let $B(n) = A(n)A(n)^\dag$ and assume all of the eigenvalues of $B(n)$ are constant.
Then the characteristic polynomial of $B(n)$ has constant coefficients.
In particular there exists a nonzero polynomial $p(t)\in\bbc[t]$ satisfying $p(B(n)) = 0$.

Assume that $\mathfrak D$ is not a constant differential operator, i.e. $\mathfrak D\notin M_N(\bbc)$.
Then we may write
$$\mathfrak D = \sum_{j=0}^\ell \partial_x^j A_j(x),\ \text{for some $\ell>0$ and $A_j(x)\in M_N(\bbc[x])$ with $A_\ell(x)\neq 0$}.$$
Since $W(x)$ is Hermitian and positive-definite on its support, the product $A_\ell(x)W(x)A_\ell(x)^*$ is not identically zero.
Hence the operator $\mathfrak D\mathfrak D^\dag$ is $W$-symmetric of order $2\ell$ with leading coefficient $A_\ell(x)W(x)A_\ell(x)^*W(x)^{-1}$.
Furthermore, we may factor $W(x) = T(x)T(x)^*$ for some $T(x)$ smooth on the support of $W(x)$.
Then the leading coefficient of $T(x)^{-1}\mathfrak D\mathfrak D^\dag T(x)$ is $(T(x)^{-1}A_\ell(x)T(x))(T(x)^{-1}A_\ell(x)T(x))^*$.
In particular it is nonzero and Hermitian and therefore not nilpotent.
Hence for all $j$ the operator $(T(x)^{-1}\mathfrak D\mathfrak D^\dag T(x))^j$, has order $2\ell j$.
Thus so too does the operator $\mathfrak D\mathfrak D^\dag$.
In particular $\mathfrak D\mathfrak D^\dag$ cannot be a root of a nonzero polynomial $p(t)\in\bbc[t]$.
This is a contradiction, so $\mathfrak D$ and $A(n)$ must be constant.
\end{proof}

As a consequence of the previous lemma, we may show that the discrete part $\mathcal D_{disc}(W)$ of $\mathcal D(W)$ is finite dimensional.
\begin{lem}
For all $i$, the algebra $\mathfrak E_i \mathcal D(W)\mathfrak E_i$ is finite dimensional.  Consequently $\mathcal D_{disc}(W)$ is finite dimensional.
\end{lem}
\begin{proof}
For brevity, let $\mathcal A\subseteq M_N(\bbc[n])$ denote the preimage of $\mathfrak E_i\mathcal D(W)\mathfrak E_i$ under the generalized Fourier map $b_P$ and let $E(n)$ be the preimage of $\mathfrak E_i$.  Note that $E(n)$ is the identity element of $\mathcal A$ and that the center of $\mathfrak A$ is $\bbc E(n)$.

The algebra $\mathcal A$ has an action on the $\bbc[n]$-submodule $E(n)\bbc[n]^{\oplus N}$ of $\bbc[n]^{\oplus N}$.
Since $\bbc[n]$ is a PID, we know that $E(n)\bbc[n]^{\oplus N}$ is isomorphic to $\bbc[n]^{\oplus m}$ for some integer $m$.
The action of $\mathcal A$ on $E(n)\bbc[n]^{\oplus N}$ is faithful, and therefore induces an algebra monomorphism $\pi: \mathcal A\rightarrow M_m(\bbc[n])$.

Since $\mathcal D(W)$ is a semiprime PI algebra, so too is $\mathcal A$.
Now suppose that $A(n)\in \mathcal A$, let $\lambda\in\bbc$ and set $B(n) = A(n)A(n)^\dag$.
By \cite[Theorem 2]{rowen1973b} the two-sided ideal of $\mathcal A$ generated by $\lambda E(n)-B(n)$ must intersect nontrivially with the center of $\mathcal A$.
Therefore there exists $C(n),\wt C(n)\in \mathcal A$ such that $C(n)(B(n)-\lambda E(n))\wt C(n) = E(n)$.
This implies that for all $\lambda$ the determinant $\det(\pi(B(n)-\lambda E(n)))$ lies in $\bbc^\times$.
Therefore the characteristic polynomial of $B(n)$ satisfies
$$p_B(t,n) := \det(tI-B(n)) = t^{N-m}\det(\pi(t E(n)-B(n)))\in \bbc[t].$$
In particular the matrix $B(n)$ has constant eigenvalues.
By Lemma \ref{super lemma} this implies that $A(n)$ is constant.
Thus we conclude all of the matrices in $\mathcal A$ are constant, and so $\mathcal A$ is finite dimensional.
\end{proof}

\begin{thm}
The algebra $\mathcal D(W)$ is affine.
\end{thm}
\begin{proof}
From the previous lemma, it suffices to show that $\mathcal D_{cont}(W)$ is finitely generated.
Thus without loss of generality, we may assume that the center of $\mathcal D(W)$ consists of a disjoint union of (possibly singular) algebraic curves over $\bbc$.
Let $\mathcal A$ be the preimage of $\mathcal D(W)$ in $M_N(\bbc[n])$ under the generalized Fourier map and let $\mathcal B=\vspan_{\bbc(n)}(\mathcal A)$.
Then $\mathcal B$ is finite dimensional over $\bbc(n)$ and has no nontrivial two-sided nilpotent ideals.
Therefore $\mathcal B$ is semisimple and by the Artin-Wedderburn theorem $\mathcal B$ is isomorphic to a direct sum
$$\mathcal B \cong \bigoplus_{j=1}^\ell M_{n_j}(\mathcal L_j)$$
for some skew-field extensions $\mathcal L_j$ of $\bbc(n)$.
Note that by Tsen's theorem, each of the $\mathcal L_j$ is actually simply a field extension of $\bbc(n)$.
For each $j$ let $R_j$ denote the integral closure of $\bbc[n]$ in $\mathcal L_j$.
Let $\mathcal L$ be a finite field extension of $\bbc(n)$ containing each of the $\mathcal L_j$'s and let $\mathcal R$ be the integral closure of $\bbc[n]$ in $\mathcal L$
and define $\Delta = \bigoplus_{j=1}^\ell M_{n_j}(\mathcal L)$.

Let $\mathcal O$ be the $\mathcal R$-subalgebra of $\Delta$ generated by $\mathcal A$.
Then $\vspan_L(\mathcal O)\subseteq \Delta$ contains $\vspan_L(\vspan_F(\mathcal A)) = \vspan_L(\bigoplus_{j=1}^\ell M_{n_j}(\mathcal L_j)) = \Delta$.
Therefore $\vspan_L(\mathcal O) = \Delta$.
Furthermore every element of $\mathcal A$ is integral over $\bbc[n]$ and therefore integral over $\mathcal R$.
Therefore every element of $\mathcal A$ is integral over $\mathcal R$, and it follows that every element of $\mathcal O$ is integral over $\mathcal R$.
Therefore by \cite[Theorem 10.3]{reiner} $\mathcal O$ is an $\mathcal R$-order in $\Delta$, and in particular $\mathcal O$ is affine and Noetherian.

The center of $\mathcal A$ maps into the center of $\mathcal B$, the latter being $\oplus_{j=1}^\ell \mathcal L_jI_{n_j}$.
Every element of $\mathcal A$ is actually integral over $\bbc[n]$, so the image of the center of $\mathcal A$ is isomorphic to a subalgebra $\mathcal Z$ of $\wt Z := \oplus_{j=1}^\ell \mathcal R I_{n_j}$.
Since $\mathcal Z$ is Noetherian and $\mathcal D = \mathcal D_{cont}(W)$, the spectrum of $\mathcal Z$ consists of a collection of affine curves over $\bbc$.
The inclusion $\mathcal Z\rightarrow\oplus_{j=1}^\ell\mathcal R_j I_{n_j}$ is precisely the normalization of $\mathcal Z$, while the inclusion maps $\mathcal R_j\rightarrow\mathcal R$ correspond finite morphisms of smooth curves.
Therefore the inclusion $\mathcal Z\subseteq\mathcal Z'$ corresponds to the finite morphism of affine algebraic curves $\spec(\mathcal Z')\rightarrow\spec(\mathcal Z)$, so in particular $\mathcal Z'$ is finitely generated as a module over $\mathcal Z$.
As a consequence, $\mathcal O$ is finitely generated as a module over $\mathcal A$.
The algebra $\mathcal O$ is an $\mathcal A$-centralizing extension, so by Montgomery and Small's extension of the Artin-Tate lemma \cite[Proposition 2]{montgomery}, $\mathcal A$ must be affine.
\end{proof}

\begin{thm}[Casper \cite{casper2017bispectral}]
The algebra $\mathcal{D}(W)$ is Noetherian and is finitely generated as a module over its center $\mathcal{Z}(W)$.
\end{thm}
\begin{proof}
Consider the preimage $\mathcal E(W)$ of $\mathcal D(W)$ under the generalized Fourier map $b_P$.
The algebra $\mathcal E(W)$ is a subalgebra of $M_N(\bbc[n])$, and thus has GK dimension at most the GK dimension of $M_N(\bbc[n])$.
The algebras $M_N(\bbc[n])$ and $\bbc[n]$ have equal GK dimension by \cite[Proposition 3.2.7]{mcconnell2001},  
and $\bbc[n]$ has GK dimension $1$, e.g. by \cite[Theorem 8.2.14]{mcconnell2001}.
Thus $\mathcal E(W)$ has GK dimension at most $1$.  Since $\mathcal D(W)$ is affine and semiprime of GK dimension at most $1$, we can apply the
Small-Stafford-Warfield theorem \cite[Theorem 1.6]{small1985} to obtain that $\mathcal D(W)$ is Noetherian and finitely generated over its center.
\end{proof}

\subsection{Localizing $\mathcal D(W)$ on its center}
Recall that for a commutative ring with zero divisors, we can still define a \vocab{total ring of fractions} to be the localization of the ring at the multiplicative set of elements in $\mathcal Z(W)$ which are not zero divisors.
Geometrically the total ring of fractions is isomorphic to the product of the ring of fractions of every irreducible component of the spectrum of the ring.
\begin{defn}
Let $W(x)$ be a weight matrix, let $P(x,n)$ be the associated sequence of monic orthogonal polynomials, and let $b_P: \mathcal{D}(W)\rightarrow M_N(\bbc[n])$ denote the generalized Fourier map.  We define $\mathcal{Z}(W)$ to be the center of $\mathcal{D}(W)$ and $\mathcal{F}(W)$ to be the total ring of fractions of $\mathcal{Z}(W)$.
\end{defn}
We will show that when we localize $\mathcal D(W)$ by extending scalars in the ring $\mathcal D(W)$ to the total ring of fractions $\mathcal F(W)$ of $\mathcal{Z}(W)$, 
we get a product of matrix algebras over the fraction fields of each of the irreducible components.
Equivalently, this means that $\mathcal D(W)$ is generically isomorphic to a matrix algebra over each irreducible component of $\mathcal Z(W)$.
This result is analogous to Posner's theorem for prime PI rings. 
\begin{thm}\label{localization theorem}
Suppose that $W(x)$ is a weight matrix.
Then
\begin{equation}\label{localization equation}
\mathcal{D}(W) \otimes_{\mathcal Z(W)} \mathcal{F}(W) \cong \bigoplus_{i=1}^r M_{n_i}(\mathcal F_i(W)),
\end{equation}
where $\mathcal F_i(W)$ is the localization of $\mathcal Z(W)$ at the minimal prime $\mathcal P_i(W)$ of $\mathcal Z(W)$, i.e., 
the function field of the irreducible curve $\mathrm{Spec} (\mathcal Z(W) / \mathcal P_i(W))$.
Moreover the localization map $\mathcal D(W)\rightarrow \mathcal D(W)\otimes_{\mathcal Z(W)}\mathcal F(W)$ is injective.
\end{thm}
\begin{remk}
The proof of the theorem actually establishes the following stronger statement: 

\emph{Any affine $*$-algebra with proper involution, which is finitely generated as a module over its center,  
localizes over its center to a product of central simple algebras.} 

The center of such an algebra is affine by the Artin-Tate Lemma \cite[p. 116]{BG}.

Rowen \cite[Appendix]{rowen1973} has previously obtained results which are similar in nature.
Specifically he proved that quasi-Goldie, semiprime, PI-algebras localize over their centers to semisimple Artinian rings.
Here quasi-Goldie means every set of independent ideals is finite and every set of left annihilators of ideals contains 
a maximal element. We do not know of the exact relation between these assumptions and ours.
\end{remk}
\begin{proof}
First we show that the natural map $\mathcal D(W)\rightarrow\mathcal D(W)\otimes_{\mathcal Z(W)} \mathcal F(W)$ is injective.
To see this, suppose that $\mathfrak D\in\mathcal D(W)$ maps to $0$ in the localization.
Then there must exist an element $\mathfrak D'\in\mathcal Z(W)$ which is not a zero divisor in $\mathcal Z(W)$ and which satisfies $\mathfrak D\mathfrak D' = 0$.
It follows that $\mathfrak D'$ annihilates the two-sided ideal of $\mathcal D(W)$ generated by $\mathfrak D$.
However, since $\mathcal D(W)$ is a semiprime PI algebra \cite[Theorem 2]{rowen1973b} tells us any nonzero two-sided ideal of $\mathcal D(W)$ must intersect nontrivially with $\mathcal Z(W)$.
This contradicts the fact that $\mathfrak D'$ is not a zero divisor in $\mathcal Z(W)$, and thus the localization map is injective.

Next note that
$$\mathcal F(W) = \mathcal F_1(W)\times\dots\times \mathcal F_r(W).$$
Using this, we have
$$\mathcal{D}(W) \otimes_{\mathcal Z(W)} \mathcal{F}(W)
= \prod_{i=1}^r \mathcal R_i,\ \ \text{for}\ \ \mathcal R_i := \mathcal{D}(W) \otimes_{\mathcal Z(W)} \mathcal{F}_i(W).$$
We know that $\mathcal D(W)$ is finitely generated as a module over its center, so the localizations $\mathcal R_i$ will also be finitely generated over $\mathcal F_i(W)$.
The latter is a field, so each of the algebras $\mathcal R_i$ is a finite dimensional $\mathcal F_i(W)$ vector space.
Note that since $\mathcal R_i$ is a localization, all elements in $\mathcal R_i$ are of the form $\mathfrak D\otimes\mathfrak A$ for some $\mathfrak D\in \mathcal D(W)$ and $\mathfrak A\in F_i(W)$.

The adjoint operation $\dag$ fixes $\mathcal Z(W)$ and extends to an adjoint operation on $\mathcal F(W)$.
Since the adjoint operation fixes each minimal prime ideal $\mathcal P_i(W)$, it also preserves the complement $\mathcal Z(W)\diff\mathcal P_i(W)$, 
and thus induces an adjoint operation on $\mathcal F_i(W)$.
This in turn induces an adjoint operation on $\mathcal R_i$, defined by $(\mathfrak D \otimes \mathfrak A)^\dag =  \mathfrak{D}^\dag  \otimes \mathfrak{A}^\dag$ for 
$\mathfrak D\in\mathcal D(W)$ and $\mathfrak A\in \mathcal F_i(W)$.

We claim that the property (\ref{nondegenerate star}) is maintained under this extension.
To see this, suppose that $(\mathfrak D\otimes\mathfrak A)(\mathfrak D\otimes\mathfrak A)^\dag = 0$ for some 
$\mathfrak D\in\mathcal D(W)$ and $\mathfrak A\in \mathcal F_i(W)$.
Note that $(\mathfrak D\otimes\mathfrak A)(\mathfrak D\otimes\mathfrak A)^\dag = (\mathfrak D\mathfrak D^\dag)\otimes (\mathfrak A\mathfrak A^\dag)$, and this latter element is zero if and only if there exists $\wt{\mathfrak D}\in \mathcal Z(W)\diff \mathcal P_i$ with $(\mathfrak D\mathfrak D^\dag)\wt{\mathfrak D} = 0$.
This then would imply that $\mathfrak D\wt{\mathfrak D}(\mathfrak D\wt{\mathfrak D})^\dag = 0$ and therefore that $\mathfrak D\wt{\mathfrak D} = 0$.  Hence $\mathfrak D\otimes\mathfrak A= 0$.  Thus the desired property holds.

Each $\mathcal R_i$ is a $*$-algebra (over $\mathcal F_i(W)$) with proper involution, therefore it is a semiprime PI-algebra.
In particular, $\mathcal R_i$ has no nonzero two-sided nilpotent ideals.
Since it is finite dimensional over $\mathcal F_i(W)$, this implies the Jacobson radical of $\mathcal R_i$ is trivial, so that $\mathcal R_i$ is a finite dimensional 
semisimple algebra over $\mathcal F_i(W)$.
Furthermore, by virtue of the construction of the $\mathcal R_i$ and the injectivity of the localization map $\mathcal D(W)\rightarrow \mathcal D(W)\otimes_{\mathcal Z(W)}\mathcal F(W)$, the center of $\mathcal R_i$ is $\mathcal F_i(W)$.
Thus $\mathcal R_i$ is a central simple algebra with center $\mathcal F_i(W)$, and by Tsen's theorem is therefore isomorphic to $M_{n_i}(\mathcal F_i(W))$ for some integer $n_i$.
This completes the proof.
\end{proof}

\begin{defn}
Suppose that $W(x)$ is a weight matrix with $\mathcal D(W)$ containing an operator of positive order.
Using the notation of Theorem \ref{localization theorem}, we call $n_1+n_2+\dots+n_r$ the \vocab{module rank} of $\mathcal{D}(W)$ over $\mathcal{Z}(W)$.
\end{defn}

\begin{prop}
Suppose that $W(x)$ is a weight matrix with $\bbc I\subsetneq \mathcal{Z}(W)$.  Then the module rank of $\mathcal{D}(W)$ is at most $N$.
\end{prop}
\begin{proof}
Let $m$ be the rank  of $\mathcal D(W)$.
Then we may choose $m$ non-nilpotent elements $\mathfrak V_1,\dots,\mathfrak V_m\in \mathcal D(W)$ such that $\mathfrak V_i\mathfrak V_j = 0$ for $i\neq j$.
Let $A_j(n)\in M_N(\bbc[n])$ be the image of $\mathfrak V_j$ under the generalized Fourier map.
Then each $A_j(n)$ is nonzero and satisfies $A_j(n)A_k(n)=0$ for $j\neq k$.
Since each $\mathfrak V_j$ is non-nilpotent, each $A_j(n)$ is also non-nilpotent.
Therefore we may choose a column vector $\vec v_j(n)$ of $\vec A_j(n)$ with $\vec A_j(n)\vec v_j(n)\neq \vec 0$.
With this choice, $\vec A_j(n)\vec v_k(n) = \vec 0$ if and only if $j\neq k$ and thus the $m$ vectors $\vec v_1(n),\dots\vec v_m(n)$ in $\bbc[n]^{\oplus N}$ are $\bbc[n]$-linearly independent.
We conclude that $m\leq N$.
\end{proof}

\begin{defn}
We say $\mathcal{D}(W)$ is \vocab{full} if the module rank of $\mathcal{D}(W)$ is $N$.
\end{defn}

Fullness is a relatively weak condition, especially for $N$ small as we see in the next theorem.
\begin{thm}
Suppose that $W(x)$ is a $2\times 2$ matrix and that $\mathcal D(W)$ is noncommutative.
Then the associated algebra of differential operators $\mathcal D(W)$ is full.
In fact in this case $\mathcal F(W)$ is a field and $\mathcal D(W)$ is a $\mathcal Z(W)$-order in $M_2(\mathcal F(W))$.
\end{thm}
\begin{proof}
Since $\mathcal D(W)$ is noncommutative, so too is its localization over its center.  Using the notation of Theorem \ref{localization theorem} and referring to (\ref{localization equation}) the noncommutativity implies that $n_i>1$ for some $i$.
However, since the rank is bounded by $2$, we get $n_1+\dots+n_r \leq 2$.  Thus $r=1$ and $n_1=2$, so that (\ref{localization equation}) says
$$\mathcal D(W) \otimes_{\mathcal Z(W)} \mathcal F(W) \cong M_2(\mathcal F(W)).$$
Thus $\mathcal D(W)$ is a $\mathcal Z(W)$-order in $M_2(\mathcal F(W))$ and in particular $\mathcal D(W)$ is full.
\end{proof}

More generally, fullness can be characterized in terms of the existence of enough orthonormal idempotents in the localization of $\mathcal D(W)$ over $\mathcal Z(W)$.  This is summarized by the next theorem.
\begin{thm}\label{idempotents theorem}
The algebra $\mathcal{D}(W)$ is full if and only if there exist $W$-symmetric elements $\mathfrak V_1, \mathfrak V_2,\dots,\mathfrak V_N$ in $\mathcal{D}(W)$ satisfying $\mathfrak V_i\mathfrak V_j = 0$ for $i\neq j$ with $\mathfrak V_1+\dots+\mathfrak V_N$ a central element of $\mathcal D(W)$ which is not a zero divisor.
\end{thm}
\begin{proof}
Suppose that $\mathcal D(W)$ is full.
Then we have an isomorphism
$$\mathcal D(W) \otimes_{\mathcal Z(W)} \mathcal F(W) \cong \bigoplus_{i=1}^r M_{n_i}(\mathcal F_i(W)),$$
with $n_1+\dots+n_r = N$.

Fix an integer $1\leq i\leq n$.
By the above, we know that the adjoint operation $\dag$ of $\mathcal D(W)$ restricts to an adjoint operation on both $\mathcal F_i(W)$ and on $M_{n_i}(\mathcal F_i(W))$.
However, the adjoint operation on $M_{n_i}(\mathcal F_i(W))$ may not be the one naturally induced by the adjoint operation on $\mathcal F_i(W)$.
In fact we in general have two adjoints on $M_{n_i}(\mathcal F_i(W))$, the adjoint coming from the restriction of $\dag$ (which we also denote by $\dag$), and the one obtained by applying $\dag$ to each of the matrix entries of an element of $M_{n_i}(\mathcal F_i(W))$ and then taking the transpose.
We denote this latter adjoint operation by $\star$.
The composition of these two adjoint operations $(\cdot)^{\dag\star}$ defines an automorphism of $M_{n_i}(\mathcal F_i(W))$ which fixes the base field $\mathcal F_i(W)$.
By the Skolem-Noether theorem, this automorphism must be inner.
Thus there exists an invertible matrix $U\in M_{n_i}(\mathcal F_i(W))$ such that
$$A^\dag = (U^\star)^{-1}A^\star U^\star, \ \ \forall A\in M_{n_i}(\mathcal F_i(W)).$$
We can apply the same argument to the opposite composition $(\cdot)^{\star\dag}$, and if we use the fact that $\dag$ and $\star$ are order two, we find
$$A^\dag = U^{-1}A^\star U, \ \ \forall A\in M_{n_i}(\mathcal F_i(W)).$$
Putting this together, we find that $U^\star = U$.
Therefore $U$ has a Cholesky-type factorization of the form $U = PDP^\star$ for some invertible matrices $P,D\in M_N(\mathcal F_i(W))$ with $D$ diagonal and $\star$-symmetric.  Then for $1\leq j\leq N$, we have
$$(P^{-1\star} E_{jj}P^{\star})^\dag = U^{-1}PE_{jj}P^{-1}U = U^{-1}PE_{jj}DP^\star = P^{-1\star}D^{-1}E_{jj}DP^\star.$$
For each $1\leq j\leq n_i$, choose $\mathfrak D_{ij}'\in \mathcal D(W) \otimes_{\mathcal Z(W)} \mathcal F(W)$ corresponding to the matrix $(P^\star)^{-1}E_{ii}P^\star$.
Then by definition $\mathfrak D_{ij}'$ is $\dag$-symmetric and satisfy $\mathfrak D_{ij}'\mathfrak D_{ab}' = 0$ for $(a,b)\neq (i,j)$.

We can repeat the above process for each $i$, obtaining $N$ elements $\mathfrak D_{ij}'$.
For each $i,j$, we can clear denominators using $W$-symmetric, central elements to obtain an element $\mathfrak D_{ij}\in \mathcal D(W)$.
By definition, the elements $\mathfrak D_{ij}$ are $W$-symmetric.
Furthermore, since the localization map is injective the $N$ elements $\mathfrak D_{ij}$ satisfy $\mathfrak D_{ij}\mathfrak D_{k\ell} = 0$ if $i\neq k$ or $j\neq\ell$, and the sum is a central element of $\mathcal D(W)$ which is not a zero divisor.

Conversely, suppose that the $\mathfrak D_1,\dots,\mathfrak D_N$ exist and set $\mathfrak D = \mathfrak D_1 + \dots + \mathfrak D_N$.
Then the $N$ elements
$$\mathfrak D^{-1}\otimes \mathfrak D_i,\ i=1,2,\dots N$$
define a family of $N$ orthogonal idempotents in $\mathcal F(W)\otimes_{\mathcal Z(W)}\mathcal D(W)$. From the isomorphism 
$$\mathcal D(W) \otimes_{\mathcal Z(W)} \mathcal F(W) \cong \bigoplus_{i=1}^r M_{n_i}(\mathcal F_i(W)),$$
we see that $\bigoplus_{i=1}^r M_{n_i}(\mathcal F_i(W))$ must have a collection of $N$ orthogonal idempotents.
This implies that $n_1+\dots+n_r$ is at least $N$ and therefore exactly $N$.
\end{proof}

\begin{prop}
Let $W(x)$ and $\wt W(x)$ be weight matrices.
Suppose that $\mathcal{D}(W)$ is full and that $\wt W(x)$ is a bispectral Darboux transformation of $W(x)$.  Then $\mathcal{D}(\wt W)$ is also full.
\end{prop}
\begin{proof}
This follows from Theorem \ref{lots of stuff}, combined with the previous theorem.
\end{proof}

\begin{defn}
Let $W(x)$ be a weight matrix.
We call a collection of $W$-symmetric elements $\mathfrak V_1,\dots, \mathfrak V_N\in \mathcal D(W)$ satisfying the condition that $\mathfrak V_i\mathfrak V_j= 0 I$ for $i\neq j$ an \vocab{orthogonal system} for $\mathcal D(W)$.
\end{defn}

Since $\mathcal D(W)$ is finitely generated over its center $\mathcal Z(W)$, it also makes sense to consider the support of a given element $\mathfrak D\in \mathcal D(W)$ over the center.
\begin{defn}
Let $W(x)$ be a weight matrix.
The \vocab{support} $\supp(\mathfrak D)$ of an element $\mathfrak D\in\mathcal D(W)$ is the set of prime ideals of $\mathcal D(W)$ containing the annihilator $\ann(\mathfrak D)$ of $\mathfrak D$ over its center, i.e.,
$$\ann(\mathfrak D) := \{\mathfrak V\in\mathcal Z(W): \mathfrak V\mathfrak D = 0\}.$$
Let $\mathcal Z_1(W),\dots,\mathcal Z_r(W)$ be the irreducible components of $\mathcal Z(W)$ and let $\mathcal P_1,\dots,\mathcal P_r$ be the corresponding minimal primes.
We say that an element $\mathfrak D\in\mathcal D(W)$ is \vocab{supported on $\mathcal Z_i(W)$} if $\mathcal P_i$ is contained in one of the primes in $\supp(\mathfrak D)$.
We say that two elements $\mathfrak D_1,\mathfrak D_2\in\mathcal D(W)$ are \vocab{supported on the same irreducible component} if for some $i$ both $\mathfrak D_1$ and $\mathfrak D_2$ are supported on $\mathcal Z_i(W)$.
\end{defn}
Alternatively, one may define the support of $\mathfrak D\in\mathcal D(W)$ to be the set of prime ideals $\mathcal P\subseteq\mathcal Z(W)$ such that $\mathfrak D$ is nonzero in the localization of $\mathcal D(W)$ at $\mathcal P$.
The support of any element in $\mathcal D(W)$ defines a closed subscheme of $\mathcal Z(W)$.
Clearly any element of $\mathcal D(W)$ must be supported on at least one irreducible component of $\mathcal Z(W)$.

There is a simple algebraic criterion for two elements $\mathfrak D_1,\mathfrak D_2\in\mathcal D(W)$ to be supported on the same irreducible component.
The proof is a consequence of (\ref{localization equation}) and is left to the reader.
\begin{prop}
Let $W(x)$ be a weight matrix with $\mathcal D(W)$ full and suppose $\mathfrak D_1,\mathfrak D_2\in\mathcal D(W)$.
Then $\mathfrak D_1\mathcal D(W)\mathfrak D_2 \neq 0$ if and only if both share an irreducible component of $\mathcal Z(W)$ on which they are supported.
\end{prop}

\section{Consequences of the algebraic structure of $\mathcal{D}(W)$}
We demonstrate in this section that the algebraic structure of $\mathcal{D}(W)$ has strong consequences for the form of the weight matrix $W(x)$.  Most importantly, we prove the Classification Theorem \ref{classification theorem}.

We will focus exclusively on the case when $\mathcal{D}(W)$ is full.  With this in mind, throughout this section the elements $\mathfrak V_1,\dots,\mathfrak V_N$ will denote an orthogonal system for $\mathcal D(W)$, i.e., a collection of nonzero, $W$-symmetric matrix differential operators in $\mathcal{D}(W)$ with $\mathfrak V_i\mathfrak V_j=0$ for $i\neq j$ and with $\mathfrak V_1+\dots+\mathfrak V_N\in\mathcal{Z}(W)$ a non-zero divisor.
Of course, the existence of an orthogonal system is guaranteed by Theorem \ref{idempotents theorem}.

\subsection{Cyclic modules and diagonalization of matrix weights}
To understand the application of the orthogonal system to our understanding of the structure of $\mathcal D(W)$, we first consider a certain left $\Omega(x)$-module and its right counterpart.
For each $\mathfrak V_i$, we consider the left $\Omega(x)$-module
\begin{equation}\label{equation for M}
\mathcal M_i := \{\vec{\mathfrak u}\in\Omega(x)^{\oplus N}: \vec{\mathfrak u}^T\mathfrak V_j =\vec 0^T, \ \forall i\neq j\},
\end{equation}
as well as the right $\Omega(x)$-module
\begin{equation}\label{equation for N}
\mathcal N_i := \{\vec{\mathfrak w}\in\Omega(x)^{\oplus N}: \mathfrak V_j\vec{\mathfrak w} =\vec 0, \ \forall i\neq j\}.
\end{equation}
The next result proves that the modules $\mathcal M_i$ and $\mathcal N_i$ are cyclic.
\begin{thm}
\label{princip}
For each $i$, there exists $\vec{\mathfrak u_i},\vec{\mathfrak w_i}\in \Omega(x)^{\oplus N}$ with 
\begin{equation}
\mathcal M_i = \Omega(x)\vec{\mathfrak u_i}\ \ \text{and}\ \ \mathcal N_i = \vec{\mathfrak w_i}\Omega(x).
\end{equation}
\end{thm}
\begin{proof}
The ring $\Omega(x)$ is a left and right principal ideal domain.
The module $\mathcal M_i$ is a submodule of the free left module $\Omega(x)^{\oplus N}$ in this PID, and is therefore free.
Furthermore, $\mathcal M_i$ contains the transposes of all the row vectors of $\mathfrak V_i$.
Since $\mathfrak V_i$ is nonzero, this implies that $\mathcal M_i$ is nonzero.
In particular, each of the $\mathcal M_i$ has rank at least $1$ as a left $\Omega(x)$-module.

Now if $\vec{\mathfrak u}\in\mathfrak V_i\cap \mathfrak V_j$ for some $i\neq j$, then $\vec{\mathfrak u}^T\mathfrak V_k=\vec 0^T$ for all $1\leq k\leq N$.
This in turn would imply that $\vec{\mathfrak u}^T(\mathfrak V_1+\dots+\mathfrak V_N)=\vec 0^T$.
Since $\mathfrak V_1+\dots+\mathfrak V_N$ is not a zero divisor, this implies that $\vec{\mathfrak u}^T = \vec 0^T$.
Thus $\mathcal M_i\cap \mathcal M_j = \vec 0$ for $i\neq j$.
Since each of the $\mathcal M_i$ is a submodule of $\Omega(x)^{\oplus N}$, it follows immediately that $\mathcal M_i$ is free of rank $1$ for all $i$.
The statement of the lemma for $\mathcal M_i$ follows immediately follows immediately.
The same proof works for $\mathcal N_i$.
\end{proof}

We fix some choice of generators $\vec{\mathfrak u_i}$ of $\mathcal M_i$ for $i=1,\dots,N$ and define $\mathfrak U$ to be the matrix differential operator whose rows are $\vec{\mathfrak u_1}^T,\dots,\vec{\mathfrak u_N}^T$
\begin{equation}\label{equation for U}
\mathfrak U = [\vec{\mathfrak u_1}\ \vec{\mathfrak u_2}\ \dots\ \vec{\mathfrak u_N}]^T.
\end{equation}
Each of the $\vec{\mathfrak u_i}$ may be written as
\begin{equation}
\vec{\mathfrak u_i} = \sum_{j=0}^{\ell_i}\partial_x^j\vec u_{ji}(x)
\end{equation}
for some $\vec u_{0i}(x),\dots,\vec u_{\ell_ii}(x)\in \bbc(x)^{\oplus N}$ with $\vec u_{\ell_ii}(x)$ not identically zero.
We define $U(x)$ to be the matrix whose rows are $\vec u_{\ell_11}(x),\dots,\vec u_{\ell_NN}(x)$
\begin{equation}\label{equation for U(x)}
U(x) = [\vec u_{\ell_11}(x)\ \vec u_{\ell_22}(x)\ \dots\ \vec u_{\ell_NN}(x)]^T.
\end{equation}
We will see below that $U(x)$ is a unit in $M_N(\bbc(x))$.

For all $i$, we know that $\vec{\mathfrak u_i}^T\mathfrak V_i\mathfrak V_j=\vec 0^T$ for $i\neq j$.
Moreover, by definition $\vec{\mathfrak u_i}^T\mathfrak V_j = \vec 0^T$ for $i\neq j$.
This means that $(\vec{\mathfrak u_i}^T\mathfrak V_i)^T\in \mathcal M_i$ for all $i$, and consequently there exists $\mathfrak v_i\in\Omega(x)$ such that
\begin{equation}\label{eigenvalue equation}
\mathfrak v_i\vec{\mathfrak u_i}^T = \vec{\mathfrak u_i}^T\mathfrak V_i\ \text{and}\ \vec{\mathfrak u_i}^T\mathfrak V_j = \vec 0^T, \ \ \forall i\neq j.
\end{equation}
We will let $m_i$ be the order of $\mathfrak v_i$ and write
\begin{equation}\label{equation for v}
\mathfrak v_i = \sum_{j=0}^{m_i}\partial_x^j v_{ji}(x)
\end{equation}
for some rational functions $v_{ji}(x)$.
\begin{remk}
The principality of certain annihilator ideals is reminiscent of Rickart $*$-rings \cite[p. 12, Definition 2]{berberian1972}.
While $\mathcal D(W)$ itself is not a Rickart $*$-ring, as it contains (left) annihilators not generated by idempotents, it would be interesting to know what properties $\mathcal D(W)$ might share with more general families of $*$-algebras.
More specifically, it would be interesting to know more about the ideal structure of $\mathcal D(W)$ in general.
\end{remk}

One of the more immediate consequences of the fullness of $\mathcal D(W)$ is that there exists a \emph{rational} matrix which diagonalizes the bilinear form associated to $W(x)$.
\begin{thm}
Let $W(x)$ be a weight matrix with $\mathcal D(W)$ full and let $\mathfrak V_1,\dots, \mathfrak V_N$, $\mathfrak U$ and $U(x)$ be defined as above.
The matrix differential operator $\mathfrak UW(x)\mathfrak U^*$ is diagonal with leading coefficient
\begin{equation}
\diag(r_1(x),\dots,r_N(x)) = U(x)W(x)U(x)^*\ \text{where}\ r_i(x) = \vec u_{\ell_ii}(x)^TW(x)\vec u_{\ell_ii}(x)^{T*}.
\end{equation}
\end{thm}
\begin{proof}
Applying (\ref{eigenvalue equation}), we see
$$\mathfrak U\mathfrak V_i = E_{ii}\mathfrak v_i\mathfrak U.$$
Now since $\mathfrak V_i$ is $W$-symmetric, this also tells us
$$\mathfrak V_iW(x)\mathfrak U^* = W(x)\mathfrak U^*E_{ii}\mathfrak v_i^*.$$
Multiplying both sides of the equality on the left by $\mathfrak U$, this says that for all $i$
$$\mathfrak v_iE_{ii}\mathfrak UW(x)\mathfrak U^* = \mathfrak UW(x)\mathfrak U^*E_{ii}\mathfrak v_i^*.$$
In particular, $\mathfrak UW(x)\mathfrak U^*$ must be a diagonal matrix
$$\mathfrak UW(x)\mathfrak U^* = \diag(\mathfrak b_1,\dots,\mathfrak b_N),\ \text{where}\ \mathfrak b_i = \vec{\mathfrak u}_i^TW(x)\vec{\mathfrak u}_i^{T*}.$$
Since $W(x)$ is Hermitian and positive-definite on $(x_0,x_1)$, we know that the function $r_i(x) := \vec u_{\ell_ii}(x)^TW(x)\vec u_{\ell_ii}(x)^{T*}$ is not identically zero.
Consequently the differential operator $\mathfrak b_i$ is order $2\ell_i$ with leading coefficient $r_i(x)$ for all $i$.
In particular
$$U(x)W(x)U(x)^* = R(x) := \diag(r_1(x),\dots,r_N(x)).$$
\end{proof}
\begin{remk}
The $U(x)$ and $r_i(x)$ are not quite the same as the $T(x)$ and $f_i(x)$ from (\ref{structural equation}).
In particular the $r_i(x)$ do not need to be classical weights.
However, as we will see below the $r_i(x)$ will be classical weights up to multiplication by some rational functions.
\end{remk}

Note that in the proof of the previous theorem, we also showed the following symmetry-type result for the $\mathfrak v_i$'s.
\begin{equation}\label{symmetry equation}
\mathfrak v_i\mathfrak b_i = \mathfrak b_i\mathfrak v_i^*.
\end{equation}
We also showed above that $\vec{\mathfrak u_i}^TW(x)\vec{\mathfrak u_j}^{T*} = 0$ for $i\neq j$, so that $\vec u_{\ell_ii}(x)^TW(x)\vec u_{\ell_jj}(x)^{T*} = 0$ for $i\neq j$.
Since $W(x)$ is Hermitian and positive-definite on $(x_0,x_1)$, it follows that $\vec u_{\ell_11}(x),\dots,\vec u_{\ell_NN}(x)$ are $\bbc(x)$-linearly independent.
In particular $\det(U(x))$ is not identically zero so that $U(x)$ is a unit in $M_N(\bbc(x))$.

\subsection{Control of the size of Fourier algebras}
As a consequence of the previous theorem, we find that the involution $\dag$ preserves a large subalgebra of $M_N(\Omega(x))$.
Moreover, we can deduce that the right Fourier algebra of the orthogonal polynomials of $W(x)$ is very large.

To see this, first note that the operators $\mathfrak v_1,\dots,\mathfrak v_N$ may be related to one another, and these relationships are encoded by the center $\mathcal Z(W)$ of $\mathcal D(W)$.
\begin{prop}
Let $W(x)$ be a weight matrix with $\mathcal D(W)$ full and let $\mathfrak V_1,\dots, \mathfrak V_N$, $\mathfrak U$, $\mathfrak v_1,\dots,\mathfrak v_N$ and $U(x)$ be defined as above.
If $\mathfrak V_i$ and $\mathfrak V_j$ are supported on the same irreducible component of $\mathcal Z(W)$, then $\mathfrak v_i$ and $\mathfrak v_j$ are Darboux conjugates and $r_i(x)/r_j(x)$ is rational.
\end{prop}
\begin{proof}
Suppose $\mathfrak V_i$ and $\mathfrak V_j$ are supported on the same irreducible component of $\mathcal Z(W)$.
Then there exists $\mathfrak D\in\mathcal D(W)$ such that $\mathfrak V_i\mathfrak D\mathfrak V_j\neq 0$.
Without loss of generality, we may take $\mathfrak D$ to be $W$-symmetric.
Moreover $\vec{\mathfrak u_k}^T\mathfrak V_i\mathfrak D\mathfrak V_j\mathfrak V_\ell\neq \vec 0^T$ only if $k=i$ and $j=\ell$, and therefore
$$\mathfrak U\mathfrak V_i\mathfrak D\mathfrak V_j = E_{ij}\mathfrak d\mathfrak U$$
for some $\mathfrak d\in\Omega(x)$.
Similarly
$$\mathfrak U\mathfrak V_j\mathfrak D\mathfrak V_i = E_{ji}\wt{\mathfrak d}\mathfrak U$$
for some $\wt{\mathfrak d}\in \Omega(x)$.

Since $\mathfrak V_i\mathfrak D\mathfrak V_j\neq 0$, we know that $\mathfrak d\neq 0$.
Now since $\mathfrak V_1+\dots+\mathfrak V_N$ is in the center of $\mathcal Z(W)$ we know that $\mathfrak V_i\mathfrak D\mathfrak V_j$ commutes with $\mathfrak V_1+\dots+\mathfrak V_N$ and therefore $\mathfrak v_1E_{11}+\dots+\mathfrak v_NE_{NN}$ commutes with $E_{ij}\mathfrak d$.
Consequently $\mathfrak v_i\mathfrak d = \mathfrak d\mathfrak v_j$ and this shows that $\mathfrak v_i$ and $\mathfrak v_j$ are Darboux conjugates.

Next, using the fact that $\mathfrak D$, $\mathfrak V_i$ and $\mathfrak V_j$ are all $W$-symmetric we calculate
$$\wt{\mathfrak d}\mathfrak b_i = \mathfrak b_j\mathfrak d^*.$$
This in particular implies that $\wt{\mathfrak d}$ is nonzero.
Comparing leading coefficients, we see that $r_i(x)/r_j(x)$ must be rational.
This completes the proof.
\end{proof}

As the previous proposition shows, $r_i(x)/r_j(x)$ will be rational for some values of $i,j$.
We can actually say more than this in the case that $W(x)$ is irreducible.
\begin{defn}
A weight matrix $W(x)$ is \vocab{reducible} if there exists a nonsingular constant matrix $A\in M_N(\bbc)$ such that $AW(x)A^*$ is a direct sum of matrix weights of smaller size.
\end{defn}
As shown in \cite{tirao2018} a weight matrix $W(x)$ is reducible if and only if $\mathcal D(W)$ contains a constant idempotent matrix different from $0$ and $I$.
Since our interest is in the classification of weight matrices, it makes sense to focus specifically on irreducible weight matrices.
\begin{prop}
Suppose that $W(x)$ is an irreducible weight matrix with $\mathcal D(W)$ full.
Then for all $i,j$ the ratio $r_i(x)/r_j(x)$ is a rational function.
\end{prop}
\begin{proof}
Suppose that for some $i$ there exists a $j$ such that the ratio $r_i(x)/r_j(x)$ is not a rational function.
Define $S\subseteq\{1,\dots,N\}$ by
$$S = \{j: r_i(x)/r_j(x)\ \text{is not rational}\}.$$
Then by assumption both $S$ and $S' := \{1,\dots,N\}\diff S$ are nonempty.
Furthermore if $j\in S$ and $k\in S'$ then $r_j(x)/r_k(x)$ is not rational.
Define $R(x) = \text{diag}(r_1(x),\dots, r_N(x))$ and note that $W(x) = U(x)^{-1}R(x)(U(x)^{-1})^*$.
Then for any $W$-symmetric $\mathfrak D\in\mathcal D(W)$ we calculate
$$R(x)(U(x)\mathfrak D U(x)^{-1})^* = (U(x)\mathfrak D U(x)^{-1})R(x).$$
Clearly $U(x)\mathfrak D U(x)^{-1}\in M_N(\Omega(x))$.
Comparing leading coefficients, we see that if the $j,k$'th coefficient of $(U(x)\mathfrak D U(x)^{-1})$ is nonzero, then $r_j(x)/r_k(x)$ must be rational.
Hence for all $j\in S$ and $k\in S^\dag$ the $j,k$'th entry of $U(x)\mathfrak DU(x)^{-1}$ must be zero.
It follows immediately that the nonzero idempontent matrix-valued rational function
$$G(x) := U(x)^{-1}\left(\sum_{j\in S} E_{jj}\right)U(x)$$
commutes with $\mathfrak D$ for all $\mathfrak D\in\mathcal D(W)$.
Since $\mathcal D(W)$ is full, this implies $G(x)\in \mathcal D(W)$ and therefore $G(x)$ is a nonzero constant idempontent matrix.
The matrix $G(x)$ is also singular so $G(x)\neq I$.
One easily checks that $G(x)^\dag = G(x)$, so this contradicts the assumption that $W(x)$ was irreducible.
This completes the proof.
\end{proof}

The next theorem shows that the right Fourier algebra is significantly large.
\begin{thm}\label{right fourier calculation}
Let $W(x)$ be an $N\times N$ irreducible weight matrix with $N>1$ and $\mathcal D(W)$ full, and let $P(x,n)$ be the associated sequence of monic orthogonal polynomials.
Then for all $\mathfrak D\in M_N(\Omega(x))$ we have $\mathfrak D^\dag\in M_N(\Omega(x))$ and there exists a polynomial $q(x)\in\bbc[x]$ with $q(x)\mathfrak D,\mathfrak Dq(x)\in\fourr(P)$.
\end{thm}
\begin{proof}
Let $\mathfrak V_1,\dots, \mathfrak V_N$, $\mathfrak U$, $U(x)$, and $\mathfrak v_1,\dots,\mathfrak v_N$ be defined as above.
Note that if $\vec u_{\ell_ii}(x)^T$ times the leading coefficient of $\mathfrak V_i$ is identically zero, then $U(x)$ times the leading coefficient of $\mathfrak V_i$ is identically zero.  
Since $U(x)$ is a unit in $M_N(\bbc(x))$, this is impossible.
Thus the order $m_i$ of $\mathfrak v_i$ must agree with the order of $\mathfrak V_i$.
Furthermore since $W(x)$ is irreducible the order of $\mathfrak V_i$ must be greater than $0$ for all $i$, as otherwise $\mathfrak V_i$ would define a constant idempotent matrix in $\mathcal D(W)$ different from $0$ and $I$.
Therefore $m_i>0$ for all $i$.

Then by the previous theorem, we know that
$$U(x)W(x)U(x)^* = \diag(r_1(x),\dots,r_N(x)).$$
By definition each of the $\mathfrak v_i$ is a differential operator with rational coefficients.
Recalling (\ref{equation for v}) we know
$$\mathfrak v_i = \sum_{j=0}^{m_i}\partial_x^j v_{ji}(x)$$
for some rational functions $v_{0i}(x),\dots,v_{m_ii}(x)\in\bbc(x)$ for all $i$, with $v_{m_ii}(x)\neq 0$.
Additionally, (\ref{symmetry equation}) tells us
$$\mathfrak b_i\mathfrak v_i^* = \mathfrak v_i\mathfrak b_i$$
where here $\mathfrak b_i = \vec{\mathfrak u}_i^TW(x)\vec{\mathfrak u}_i^{T*}$ is a differential operator of order $2\ell_i$ with leading coefficient $r_i(x)$.
Let $s_i(x)$ represent the subleading coefficient of $\mathfrak b_i$.
Comparing coefficients, we see that
$$(-1)^{m_i}r_i(x)v_{m_ii}(x)^* = v_{m_ii}(x)r_i(x)$$
and also that
\begin{align*}
v_{m_ii}(x)s_i(x) + v_{(m_i-1)i}(x)r_i(x) + 2\ell_iv_{m_ii}'(x)r_i(x)
 & = 
(-1)^{m_i}r_i(x)(m_iv_{m_ii}'(x)^*-v_{(m_i-1)i}(x)^*)\\
 & + (-1)^{m_i}s_i(x)v_{m_ii}(x)^*\\
 & + (-1)^{m_i}m_ir_i'(x)v_{m_ii}(x)^*.
\end{align*}
Combined with the first equation, this second equation simplifies to 
\begin{align*}
(v_{(m_i-1)i}(x)+(-1)^{m_i}v_{(m_i-1)i}(x))^*r_i(x) + (2\ell_i-m_i)v_{m_ii}'(x)r_i(x)  =  m_ir_i'(x)v_{m_ii}(x).
\end{align*}
Solving this differential equation for $r_i(x)$, we obtain for $m_i$ even
\begin{equation}\label{equation for r real}
r_i(x) = v_{m_ii}(x)^{2\ell_i/m_i-1}\exp\int\frac{\Re(v_{(m_i-1)i}(x))}{(m_i/2)v_{m_ii}(x)}dx,
\end{equation}
and for $m_i$ odd
\begin{equation}\label{equation for r imag}
r_i(x) = v_{m_ii}(x)^{2\ell_i/m_i-1}\exp\int\frac{\Im(v_{(m_i-1)i}(x))}{(m_i/2)v_{m_ii}(x)}dx.
\end{equation}
In either case, conjugation by $r_i(x)$ sends $\Omega(x)$ to $\Omega(x)$.

Now suppose that $\mathfrak D\in M_N(\Omega(x))$.
Set $R(x) = \text{diag}(r_1(x),\dots,r_N(x))$.
Then we know that $W(x) = U(x)^{-1}R(x)(U(x)^{-1})^*$.
Therefore
$$\mathfrak D^\dag = U(x)^{-1}R(x)(U(x)^{-1})^*\mathfrak D^*U(x)^*R(x)^{-1}U(x)$$
has rational entries if and only if $R(x)^{-1}U(x)\mathfrak DU(x)^{-1}R(x)$ has rational entries.
By the previous proposition combined with the fact that conjugation by $r_i(x)$ preserves $\Omega(x)$, we see that $\mathfrak D^\dag$ has rational entries.

Now suppose that $\mathfrak D\in M_N(\Omega(x))$.
Choose a polynomial $q_0(x)$ such that both $q_0(x)\mathfrak D$ and $\mathfrak D^\dag q_0(x)$ have polynomial coefficients.
Let $q_1(x)$ be the unique polynomial vanishing at the finite endpoints of the support of $W(x)$.
If the order of $\mathfrak D$ is $\ell$ set $q_1(x) = q_0(x)q_1(x)^{2\ell}$.
Then both $q_1(x)\mathfrak D$ and $(q_1(x)\mathfrak D)^\dag$ are in $M_N(\Omega[x])$, and the coefficients of $q_1(x)\mathfrak D$ vanishes to sufficiently high order at the endpoints of the support of $W(x)$ so $q_1(x)\mathfrak D$ is $W$-adjointable.  Hence $q_1(x)\mathfrak D\in \fourr(P)$.  Similarly, we can choose $q_2(x)\in \bbc[x]$ with $\mathfrak Dq_2(x)\in\fourr(P)$.  Taking $q(x) = q_1(x)q_2(x)$, we get $q(x)\mathfrak D,\mathfrak Dq(x)\in\fourr(P)$.
\end{proof}

\subsection{First part of the proof of the Classification Theorem}
We next prove a theorem that comprises \emph{most} of the Classification Theorem from the introduction. 
It tells us that $\mathfrak V_1,\dots,\mathfrak V_N$ are Darboux conjugate to degree-filtration preserving differential operators of order at most two.
To do so, we first recall a result for algebras of commuting differential operators with rational spectra.
\begin{lem}[Kasman \cite{kasman1998}]\label{kasman lemma}
Let $\mathcal A\subseteq\Omega(x)$ be a commutative subalgebra with $\spec(\mathcal A)$ a rational curve.
Then there exist differential operators $\mathfrak h,\mathfrak d\in\Omega(x)$ with
$$\mathfrak h\mathcal A\mathfrak h^{-1}\subseteq \bbc[\mathfrak d]$$
where $\mathfrak d$ has order equal to the greatest common divisor of the orders of operators in $\mathcal A$.
\end{lem}

With this in mind, we have the following theorem.
\begin{thm}\label{main theorem first part}
Suppose that $W(x)$ is a weight matrix with $\mathcal D(W)$ full and that $\mathcal D(W)$ contains a $W$-symmetric second-order differential operator whose leading coefficient multiplied by $W(x)$ is positive definite on the support of $W(x)$.
Then there exist rational matrix differential operators $\mathfrak T,\wt{\mathfrak T}\in M_N(\Omega(x))$ with
\begin{equation}\label{conjugation  equation}
\mathfrak T\wt{\mathfrak T} = \diag(p_i(\mathfrak d_1),\dots,p_i(\mathfrak d_N))\ \ \text{and}\ \ \wt{\mathfrak T}E_{ii}\mathfrak T = q(\mathfrak V_i)
\end{equation}
for all $i$ where for each $i$ $\mathfrak d_i\in\Omega(x)$ is a differential operator of order $1$ or $2$ and $p_i,q$ are nonzero polynomials.
If any nonconstant polynomials in $\mathfrak v_i$ and $\mathfrak v_j$ are Darboux conjugates (for example if $\mathfrak V_i$ and $\mathfrak V_j$ are supported on the same irreducible component of $\mathcal Z(W)$) then we may take $\mathfrak d_i = \mathfrak d_j$.
\end{thm}
\begin{proof}
Suppose that $\mathfrak D\in \mathcal D(W)$.
Then we calculate for all $i,j,k$ that $\vec{\mathfrak u_i}^T\mathfrak V_j\mathfrak D\mathfrak V_j \mathfrak V_k$ is $\vec 0^T$ if $i\neq j$ or $j\neq k$.
Therefore for all $j$ there exists $\mathfrak d_i\in \Omega(x)$ satisfying
$$\mathfrak U\mathfrak V_j\mathfrak D\mathfrak V_j = \mathfrak d_jE_{jj}\mathfrak U.$$
Since $\mathfrak V_1 + \dots + \mathfrak V_N\in \mathcal Z(W)$ we also know that $\mathfrak d_j$ commutes with $\mathfrak v_j$.
Thus for all $j$ the algebra $\mathfrak V_j\mathcal D(W)\mathfrak V_j$ is isomorphic to a subalgebra $\mathcal A_j\subseteq \Omega(x)$ defined by
$$\mathcal A_j = \{\mathfrak d\in \Omega(x): \mathfrak dE_{jj}\mathfrak U = \mathfrak U\mathfrak V_j\mathfrak D\mathfrak V_j,\ \mathfrak D\in \mathcal D(W)\}.$$
Each element of $\mathcal A_j$ commutes with the operator $\mathfrak v_j$, so by a result of Schur the algebra $\mathcal A_j$ is a commutative subalgebra of $\Omega(x)$.

For all $i$ let $m_i$ be the order of $\mathfrak v_i$.
Let $\mathfrak D$ be a second-order differential operator in $\mathcal D(W)$ whose leading coefficient multiplied by $W(x)$ is positive definite on the support of $W(x)$.
For all $i$, the differential operator $\mathfrak V_i\mathfrak D\mathfrak V_i$ has order $2m_i + 2$.
To see this, write $\mathfrak V_i = \sum_{j=0}^{m_i}\partial_x^jV_{ji}(x)$ and $\mathfrak D = \partial_x^2D_2(x) + \partial_xD_1(x) + D_0(x)$ and note that since $\mathfrak V_i$ is $W$-symmetric we must have $V_{m_ii}(x)D_2(x)V_{m_ii}(x) = V_{m_ii}(x)D_2(x)W(x)V_{m_ii}(x)^*W(x)^{-1}$.
By assumption $D_2(x)W(x)$ is positive-definite and Hermitian on the support of $W(x)$.
Therefore $V_{m_ii}(x)D_2(x)V_{m_ii}(x)^*$ is not identically zero on the support of $W(x)$.
Hence $V_{m_ii}(x)D_2(x)W(x)V_{m_ii}(x)^*W(x)^{-1}$ is not identically zero so the product $\mathfrak V_i\mathfrak D\mathfrak V_i$ has the desired order.

As mentioned in the previous paragraph, there must exist $\mathfrak a_i\in\Omega(x)$ with
$$\mathfrak U\mathfrak V_i\mathfrak D\mathfrak V_i = \mathfrak a_iE_{ii}\mathfrak U.$$
In other words, for all $i$ and $j$
$$\vec{\mathfrak u_j}^T\mathfrak V_i\mathfrak D\mathfrak V_i = \delta_{ji}\mathfrak a_i\vec{\mathfrak u_i}^T.$$
If the order of $\mathfrak a_i$ is less than the order of $\mathfrak V_i\mathfrak D\mathfrak V_i$, then this would imply that $U(x)$ would be a left zero divisor of the leading coefficient of $\mathfrak V_i\mathfrak D\mathfrak V_i$.
The determinant of $U(x)$ is not identically zero, so this is impossible.
Thus the order of $\mathfrak a_i$ is $2m_i+2$.
In particular the greatest common divisor of the orders of elements in $\mathcal A_i$ is either $1$ or $2$.

We claim that $\spec(\mathcal A_j)$ is a rational curve.
To see this, for all $i$ let $\Lambda_i(n)$ be the sequence of matrices satisfying $\Lambda_i(n)P(x,n)=P(x,n)\cdot\mathfrak V_i$.
In terms of the generalized Fourier map $b_P(\Lambda_i(n)) = \mathfrak V_i$ for all $i$.
Since $\mathfrak V_i$ is $W$-symmetric, we also know that $\|P(x,n)\|_W^2\Lambda_i(n)^*\|P(x,n)\|_W^{-2} = \Lambda_i(n)$ for all $n$.
Therefore $\|P(x,n)\|_W^{-1}\Lambda_i(n)\|P(x,n)\|_W$ defines a sequence of Hermitian matrices for all $i$.
The spectral theorem tells us they are unitarily diagonalizable and since they commute they are \emph{simultaneously} unitarily diagonalizable.
Thus there exists a sequence of unitary matrices $U(n)$ and sequences of real numbers $\lambda_1(n),\dots,\lambda_N(n)$ such that for all $i$
$$C(n)^{-1}\Lambda_i(n)C(n) = \lambda_i(n)E_{ii},$$
where here $C(n) := \|P(x,n)\|_WU(n)$.

Now again take $\mathfrak D\in\mathcal D(W)$.
Let $\Theta_i(n) = b_P^{-1}(\mathfrak V_i\mathfrak D\mathfrak V_i)$.
Since $\mathfrak V_i\mathfrak D\mathfrak V_i$ commutes with $\mathfrak V_i$ for all $i$ and is annihilated by $\mathfrak V_j$ for $i\neq j$, we know that there exist sequences of complex numbers $\theta_1(n),\dots,\theta_N(n)$ with $C(n)^{-1}\Theta_i(n)C(n) = \theta_i(n)E_{ii}$.
We also know that for all $i$ the sequence $\Theta_i(n)\in M_N(\bbc[n])$.
Since $\theta_i(n) = \tr(\Theta_i(n))$, we see that $\theta_i(n)\in\bbc[n]$ for all $i$.
Thus for all $i$, the generalized Fourier map defines an injection $\mathcal A_i\rightarrow \bbc[n]E_{ii}$
via
$$\mathcal A_i\xleftarrow{\cong} \mathfrak V_i\mathcal D(W)\mathfrak V_i\xrightarrow{C(n)^{-1}b_P^{-1}(\cdot)C(n)} \bbc[n]E_{ii}.$$
Thus $\spec(\mathcal A_i)$ is unirational and by L\"{u}roth's theorem $\spec(\mathcal A_i)$ is rational.

Since $\spec(\mathcal A_i)$ is rational, the previous lemma tells us that there exist differential operators $\mathfrak h_i,\mathfrak d_i\in\Omega(x)$ with the  order of $\mathfrak d_i$ equal to the greatest common divisor of the orders of elements in $\mathcal A$ (either $1$ or $2$) such that $\mathfrak h_i\mathcal A\mathfrak h_i^{-1}\subseteq \bbc[\mathfrak d_i]$.
If any nonconstant polynomials in $\mathfrak v_i$ and $\mathfrak v_j$ are Darboux conjugates, then $\mathfrak d_i$ and $\mathfrak d_j$ are Darboux conjugates and by modifying our choice of the $\mathfrak h_i$ we can take $\mathfrak d_i=\mathfrak d_j$.

We can also obtain a revised version of \ref{symmetry equation}.
To see this, define $\wt{\mathfrak b_i} = \mathfrak h_i\mathfrak b_i\mathfrak h_i^*$.
Then for any $\mathfrak a\in\mathcal A_i$ coming from a $W$-symmetric $\mathfrak D\in \mathcal D(W)$, we have $\mathfrak b_i\mathfrak a^* = \mathfrak a\mathfrak b_i$.
Therefore $\wt{\mathfrak b_i}(\mathfrak h\mathfrak a\mathfrak h^{-1})^*\wt{\mathfrak b} = \mathfrak h\mathfrak a\mathfrak h^{-1}$.
Since $\mathcal A_i$ is spanned by such elements, this defines an involution of $\mathcal A_i$ and of $\mathfrak h\mathcal A_i\mathfrak h^{-1}$ and therefore of their fraction fields.
Since the fraction field of the subalgebra $\mathfrak h_i\mathcal A_i\mathfrak h_i^{-1}$ of $\bbc[\mathfrak d_i]$ is $\bbc(\mathfrak d_i)$ (just by rationality plus order arguments),
this means that for all $p(\mathfrak d_i)\mapsto \wt{\mathfrak b_i}p(\mathfrak d_i)^*\wt{\mathfrak b_i}^{-1}$ defines an automorphism of $\bbc(\mathfrak d_i)$.
The automorphism fixes $0,1$ and $\mathfrak h_i\mathfrak v_i\mathfrak h_i^{-1}$ and therefore must be the identity.
Thus in particular
\begin{equation}\label{symmetry equation revisited}
\wt{\mathfrak h}\mathfrak d_i^* = \mathfrak d_i\wt{\mathfrak b_i}
\end{equation}

For each $i$ there exists a polynomial $p_i(\mathfrak d_i)\in\bbc[\mathfrak d_i]$ with
$$\mathfrak h_i\mathfrak v_i=p_i(\mathfrak d_i)\mathfrak h_i.$$
This implies that
$$\ker(\mathfrak h_i^*)\cdot \mathfrak v_i\subseteq \ker(\mathfrak h_i^*).$$
Choose a polynomial $q(\mathfrak v_i)\in\bbc[\mathfrak v_i]$ such that $\ker(\mathfrak h_i^*)\cdot q(\mathfrak v_i) = 0$ for all $i$.
Then there exists a rational differential operator $\mathfrak t_i$ with
$$q(\mathfrak v_i) = \mathfrak t_i\mathfrak h_i.$$

Let $\mathfrak c_i = \mathfrak b_i^{-1}\mathfrak v_i$, viewed as a pseudo-differential operator.
Note that $\mathfrak c_i$ may actually be shown to be a differential operator, but this is not necessary for our arguments.
It's clear from Equation \ref{symmetry equation} that $\mathfrak c_i$ is $*$-symmetric.
Furthermore an easy calculation shows
$$\mathfrak V_i = W(x)\mathfrak U^*\mathfrak c_iE_{ii}\mathfrak U.$$
This means that
\begin{align*}
\mathfrak V_iq(\mathfrak V_i)
  & = W(x)\mathfrak U^*\mathfrak c_iq(\mathfrak v_i)E_{ii}\mathfrak U\\
  & = W(x)\mathfrak U^*\mathfrak c_i\mathfrak t_i\mathfrak h_iE_{ii}\mathfrak U
\end{align*}
and also that
\begin{align*}
p_i(\mathfrak d_i)q(p_i(\mathfrak d_i)) E_{ii}
  & = \mathfrak h\mathfrak v_iq(\mathfrak v_i)\mathfrak h_i^{-1} E_{ii}\\
  & = \mathfrak h\mathfrak b_i\mathfrak c_iq(\mathfrak v_i)\mathfrak h_i^{-1}E_{ii}\\
  & = \mathfrak h\mathfrak b_i\mathfrak c_i \mathfrak t_iE_{ii}\\
  & = \mathfrak h\mathfrak UW(x)\mathfrak U^*\mathfrak c_i \mathfrak t_iE_{ii}.
\end{align*}
Take 
$$\mathfrak T = \diag(\mathfrak h_1,\dots,\mathfrak h_N)\mathfrak U$$
and also
$$\wt{\mathfrak T} = W(x)\mathfrak U^*\diag(\mathfrak c_1\mathfrak t_1,\dots,\mathfrak c_N\mathfrak t_N) = \sum_i W(x)\vec{\mathfrak u_i}^*\mathfrak c_i\mathfrak t_i\vec e_i^T.$$
It's clear from the construction that $\mathfrak T\in M_N(\Omega(x))$.
Note that for all $i$ the row vectors of $\mathfrak V_i$ belong to $\mathcal M_i$, so that $\mathfrak V_i = \vec{\mathfrak a_i}\vec{\mathfrak u_i}^{T*}$ for some $\vec{\mathfrak a_i}\in\Omega(x)^{\oplus N}$.
Since the ring of pseudo-differential operators is an integral domain and $W(x)\vec{\mathfrak u_i}^{T*}\mathfrak c_i \vec{\mathfrak u_i}^T = \mathfrak V_i$, we find $W(x)\vec{\mathfrak u_i}^{T*}\mathfrak c_i = \vec{\mathfrak a_i}$.  Hence $\wt{\mathfrak T}$ is rational.
Furthermore for all $i$
$$\mathfrak T\wt{\mathfrak T} = \diag(p_1(\mathfrak d_1)q(p_1(\mathfrak d_1)),\dots,p_N(\mathfrak d_N)q(p_N(\mathfrak d_N)))$$
and also
$$\wt{\mathfrak T}E_{ii}\mathfrak T = \mathfrak V_iq(\mathfrak V_i).$$
\end{proof}

\begin{thm}\label{backup strats}
Suppose that $W(x)$ satisfies the assumptions of the previous theorem and let $p_i,q$, $\mathfrak T$, $\wt{\mathfrak T}$ and $\mathfrak d_i$ be defined as in the above theorem, with $\mathfrak d_i = \mathfrak d_j$ if and only if some nonconstant polynomials in $\mathfrak d_i$ and $\mathfrak d_j$ are Darboux conjugates.
Let $\mathfrak d_1',\dots,\mathfrak d_r'$ be the \emph{distinct} values of $\mathfrak d_1,\dots,\mathfrak d_N$, and let $n_i$ be the multiplicity of $\mathfrak d_i'$, i.e. $n_i = \#\{i: \mathfrak d_i = \mathfrak d_i'\}$.
Then the centralizer of $\mathfrak d_1\oplus\dots\oplus\mathfrak d_N$ in $M_N(\Omega[x])$ is given by
$$\mathcal C(\mathfrak d_1\oplus\dots\oplus\mathfrak d_N) \cong \bigoplus_{i=1}^r M_{n_i}(\bbc[\mathfrak d_i']).$$
Moreover conjugation by $\mathfrak T$ defines an embedding of $\mathcal D(W)$ into $\mathcal C(\mathfrak d_1\oplus\dots\oplus\mathfrak d_N)$.
\end{thm}
\begin{proof}
Suppose that $\mathfrak A\in M_N(\Omega[x])$ commutes with $\mathfrak d_1\oplus\dots\oplus\mathfrak d_N$, and let $\mathfrak a_{ij}$ be the entries of $\mathfrak A$.
For all $i$ and $j$ we must have $\mathfrak d_i\mathfrak a_{ij} = \mathfrak a_{ij}\mathfrak d_j$.  Thus by the choice of the $\mathfrak d_i$'s we see that $\mathfrak a_{ij}=0$ unless $\mathfrak d_i=\mathfrak d_j$, in which case $\mathfrak a_{ij}$ belongs to the centralizer of $\mathfrak d_j$ in $\Omega[x]$, the latter being $\bbc[\mathfrak d_j]$.
From this the above isomorphism is clear.

Fix $i,j$.
We define a map $\phi_{ij}:\mathcal D(W)\rightarrow \bbc(\mathfrak d_i)$ as follows.
By arguments already applied in previous proofs, the element $\mathfrak V_i\mathfrak D\mathfrak V_j$ satisfies $\mathfrak T\mathfrak V_i\mathfrak D\mathfrak V_j = \mathfrak aE_{ij}\mathfrak T$ for some $\mathfrak a$ satisfying $\mathfrak a p_j(\mathfrak d_j)) = p_i(\mathfrak d_i)\mathfrak a$.
If $\mathfrak a$ is nonzero then again by the choice of the $\mathfrak d_i$ we know that $\mathfrak d_i=\mathfrak d_j$ so that $\mathfrak a$ commutes with $\mathfrak p_i(\mathfrak d_i)$ and therefore with $\mathfrak d_i$.  We define $\phi_{ij}(\mathfrak D) = \mathfrak ap_i(\mathfrak d_i)^{-1}p_j(\mathfrak d_j)^{-1}$.
More generally for any $i,j$ and $\mathfrak D\in \mathcal D(W)$ we define $\phi_{ij}(\mathfrak D)\in \bbc(\mathfrak d_i)$ to be the unique element satisfying
$$p_i(\mathfrak d_i)\phi_{ij}(\mathfrak D)p_j(\mathfrak d_j)E_{ij}\mathfrak T = \mathfrak T\mathfrak D.$$
We further define
$$\Phi: \mathcal D(W)\rightarrow\bigoplus_{i=1}^r M_{n_i}(\bbc(\mathfrak d_i'))$$
by setting $\Phi(\mathfrak D)$ to be the matrix whose entries are $\phi_{ij}(\mathfrak D)$ for all $i,j$.
It is easy to see that $\Phi$ is an algebra monomorphism.
In fact, if we view $\bigoplus_{i=1}^r M_{n_i}(\bbc(\mathfrak d_i'))$ as a subalgebra of the algebra of $N\times N$ matrices with pseudo-differential operator entries, then $\Phi$ is simply the conjugation map $\Phi: \mathfrak D\mapsto \mathfrak T\mathfrak D\mathfrak T^{-1}$.

The center $\mathcal Z(W)$ of $\mathcal D(W)$ is mapped under $\Phi$ to $\bigoplus_{i=1}^r \bbc[\mathfrak d_i']I_{n_i}$ (for $I_{n_i}$ the $n_i\times n_i$ identity matrix).
In particular, the image of the center $\mathcal Z(W)$ is integral over $\bigoplus_{i=1}^r \bbc[\mathfrak d_i']I_{n_i}$ (i.e. each element is the root of a monic polynomial with coefficients n this ring).
The algebra $\mathcal D(W)$ is a finite module over its center, so we conclude that the image of $\mathcal D(W)$ is integral over $\bigoplus_{i=1}^r \bbc[\mathfrak d_i']I_{n_i}$.

We claim that $\Phi$ actually maps $\mathcal D(W)$ into $\bigoplus_{i=1}^r M_{n_i}(\bbc[\mathfrak d_i'])$.
To see this, first note that as a consequence of \ref{symmetry equation revisited}, we have $\Phi(\mathfrak D^\dag) = \Phi(\mathfrak D)^\star$, where $\star$ represents the unique involution of $\bigoplus_{i=1}^r M_{n_i}(\bbc(\mathfrak d_i'))$ extending Hermitian conjugate on $M_{n_i}(\bbc)$ and sending $\mathfrak d_i'$ to $\mathfrak d_i'$.
Thus if $\mathcal A_j$ is the image of $\mathcal D(W)\rightarrow \bigoplus_{i=1}^r M_{n_i}(\bbc(\mathfrak d_i'))\rightarrow M_{n_j}(\bbc(\mathfrak d_j'))$ then
$\mathcal A_j$ may be identified with a subalgebra of $M_{n_j}(\bbc(t))$ closed under Hermitian conjugation, with each element satisfying a monic polynomial identity with coefficients in $\bbc[t]$.
However if $F(t)\in M_{n_j}(\bbc(t))$ is Hermitian and integral over $\bbc[t]$, then $F(t)^2$ is integral over $\bbc[t]$.
This means that $\tr(F(t)^2)$ is in $\bbc[t]$.
However $\tr(F(t)^2) = \sum_{ij} |f_{ij}(t)|^2$, where $f_{ij}(t)$ are the entries of $F(t)$, so this means $F(t)$ is in $M_{n_j}(\bbc[t])$.
Since $\mathcal A_j$ is spanned by its Hermitian matrices, this shows that $\mathcal A_j\subseteq M_{n_j}(\bbc[\mathfrak d_j'])$.
This proves our claim and our theorem.
\end{proof}

\subsection{Second part of the proof of the Classification Theorem}
We now complete the proof of the Classification Theorem stated in the introduction.
Before doing so, we reqire a couple quick lemmas.
\begin{lem}
For all $i$ the function $v_{m_ii}(x)$ is a polynomial in $x$ which vanishes at the finite endpoints of the support $(x_0,x_1)$ of $W(x)$.
\end{lem}
\begin{proof}
For each $i$ we may write $\mathfrak V_i = \sum_{j=0}^{m_i}\partial_x^jV_{ji}(x)$ for some polynomials $V_{ji}(x)\in M_N(\bbc[x])$ with $V_{m_ii}(x)$ not identically zero and $m_i$ the order of $\mathfrak v_i$, as above.
Since $\mathfrak V_i$ is $W$-symmetric, the leading coefficient $V_{m_ii}(x)$ of $\mathfrak V_i$ must evaluate to a nilpotent matrix at the finite endpoints of the support of $W(x)$.
In particular, its eigenvalues at these points must be zero.
To see this, note that since $\mathfrak V_i$ is $W$-adjointable, so too is $\frac{1}{m_i!}\Ad_{xI}^{m_i-1}(\mathfrak V_i) = \partial_xV_{m_ii}(x)+\frac{1}{m_i}V_{(m_i-1)i}(x)$.
Therefore so is $(\partial_xV_{m_ii}(x))^k$ for any integer $k>0$.
By taking appropriate linear combinations for various $k$, it follows that $V_{m_ii}(x)^k\partial_x^k$ is $W$-adjointable for all $k$.
Hence for all $k$ we must have the expression $V_{m_ii}(x)^kW(x)^{(k)}$ vanish as $x$ approaches $x_0$ or $x_1$ from within $(x_0,x_1)$.
For each finite endpoint there must exist a value of $k$ such that the derivative $W(x)^{(k)}$ evaluates to a nonzero matrix at the endpoint.
Consequently $V_{m_ii}(x)^k$ must evaluate to the zero there.

Next note that the matrix $U(x)$ as defined above satisfies $U(x)V_{m_ii}(x) = v_{m_ii}(x)U(x)$ and since $U(x)$ is a unit in $M_N(\bbc(x))$, this implies that $v_{m_ii}(x)$ is equal to the trace of $V_{m_ii}(x)$.
Hence it is a polynomial and it evaluates to $0$ at the finite endpoints of the support of $W(x)$.
\end{proof}

\begin{lem}
For all $1\leq i\leq N$, let $h_i(x,n)$ be a sequence of classical orthogonal polynomials and let $P(x,n)$ be a sequence of matrix polynomials satisfying the property that
$$P(x,n) = \text{diag}(h_1(x,n-m),\dots,h_N(x,n-m))$$
for all integers $n\geq\ell$ for some fixed integers $\ell,m$ with $\ell\geq 0$.
Then for all $i$, there exists a sequence of classical orthogonal polynomials $\wt p_i(x,n)$ such that
$$\wt P(x,n) := \text{diag}(\wt p_1(x,n),\dots,\wt p_N(x,n))$$
is a bispectral Darboux transformation of $P(x,n)$.
\end{lem}
\begin{proof}
The sequences of classical orthogonal polynomials come with intertwining operators which relate polynomials of various degrees.
For example, consider the Jacobi polynomials $j_{\alpha,\beta}(x,n)$ for the weight $(1-x)^\alpha(1+x)^\beta1_{(-1,1)}(x)$.
They are eigenfunctions of the Jacobi operator $\mathfrak e_{\alpha,\beta} = \partial_x^2(1-x^2) + \partial_x(\beta-\alpha-(\beta+\alpha+2)x)$.
The intertwining operator for $j_{\alpha,\beta}(x,n)$ is defined as
$$\mathfrak t_{\alpha,\beta} := \partial_x(1-x^2) + \beta-\alpha-(\beta+\alpha+2)x.$$
Note that
$$\mathfrak e_{\alpha,\beta} = \partial_x\mathfrak t_{\alpha,\beta}\ \ \text{and}\ \ \mathfrak e_{\alpha+1,\beta+1} - (\beta+\alpha+2) = \mathfrak t_{\alpha,\beta}\partial_x$$
so $j_{\alpha+1,\beta+1}(x,n)\cdot\mathfrak t_{\alpha,\beta}$ is an polynomial eigenfunction of $\mathfrak e_{\alpha,\beta}$ of degree $n+1$ and therefore equal to a constant multiple of $j_{\alpha,\beta}(x,n+1)$.  Similarly $\partial_x$ sends $j_{\alpha,\beta}(x,n)$ to $j_{\alpha+1,\beta+1}(x,n-1)$.
It follows that
$$j_{\alpha,\beta}(x,n-m)\cdot \partial_x^\ell\mathfrak t_{\alpha+\ell-1,\beta+\ell-1}\dots\mathfrak t_{\alpha-m,\beta-m} = c_nj_{\alpha-m,\beta-m}(x,n)$$
for some sequence of constants $c_n$.
Furthermore $\partial_x^m$ sends $j_{\alpha-m,\beta-m}(x,n)$ to a constant multiple of $j_{\alpha,\beta}(x,n-m)$ and therefore
the $j_{\alpha,\beta}$ must be eigenfunctions of the product.
Thus
$$(\partial_x^\ell\mathfrak t_{\alpha+\ell-1,\beta+\ell-1}\dots\mathfrak t_{\alpha-m,\beta-m})\partial_x^m \in \bbc[\mathfrak e_{\alpha,\beta}].$$

By similar arguments to those in the previous paragraph for the Hermite and Jacobi case, we see that taking appropriate products of $\partial_x$ and the associated intertwining operators we can obtain differential operators $\mathfrak t_i,\wt{\mathfrak t_i}\in\Omega[x]$ such that the kernel of $\mathfrak t_i$ contains all polynomials of degree less than $\ell$ and such that
$$\mathfrak t_i\wt{\mathfrak t_i}\in \bbc[\mathfrak d_i]\ \ \text{and}\ \ h_i(x,n-m)\cdot\mathfrak t_i = c_i\wt h_i(x,n)$$
for some sequence of monic orthogonal polynomials $\wt h_i(x,n)$.
Note that all but finitely many of the $c_n$'s must be nonzero, since otherwise the kernel of the differential operator will be too large.
Setting $\mathfrak T = \text{diag}(\mathfrak t_1,\dots,\mathfrak t_N)$ and $\wt{\mathfrak T} = \text{diag}(\wt{\mathfrak t_1},\dots,\wt{\mathfrak t_N})$ it follows that
$$P(x,n)\cdot\mathfrak T = C(n)\wt P(x,n)\ \ \text{and}\ \ \wt P(x,n)\cdot\wt{\mathfrak T} = \wt C(n)P(x,n)$$
for some some sequences of matrices $C(n),\wt C(n)$ nonsingular for almost every $n$ and for $\wt P(x,n)$ defined as in the statement of the lemma.
By Theorem \ref{right fourier calculation}, it follows that $\wt P(x,n)$ is a bispectral Darboux transformation of $P(x,n)$.
\end{proof}

\begin{proof}[Proof of the Classification Theorem \ref{classification theorem}]
We assume $W(x)$ is a weight matrix and that $\mathcal D(W)$ contains a $W$-symmetric second-order differential operator whose leading coefficient multiplied by $W(x)$ is positive definite on the support of $W(x)$.
Without loss of generality, we may assume that $W(x)$ is unitarily irreducible $N\times N$ weight matrix with $N>1$.

The first direction is easy.  If $W(x)$ is a bispectral Darboux transformation of a direct sum of classical weights $f_1(x)\oplus \dots\oplus f_N(x)$, then the fact that $\mathcal D(f_1(x)\oplus\dots\oplus f_N(x))$ is full automatically implies $\mathcal D(W)$ is full.
The difficult part is proving the converse.
Assume $\mathcal D(W)$ is full.

By Theorem \ref{main theorem first part}, we may choose $\mathfrak T,\wt{\mathfrak T}\in M_N(\Omega(x))$ with $\wt{\mathfrak T}E_{ii}\mathfrak T = q(\mathfrak V_i)$ and with $\mathfrak T\wt{\mathfrak T} = \diag(p_1(\mathfrak d_1),\dots,p_N(\mathfrak d_N))$ for some differential operators $\mathfrak d_1,\dots,\mathfrak d_N$ of order $1$ or $2$ and polynomials $p_i,q$.
Note that by rescaling $\mathfrak T$ on the left and $\wt{\mathfrak T}$ on the right, we may without loss of generality assume that $\wt{\mathfrak T}$ has polynomial coefficients.
Also from the proof of Theorem \ref{main theorem first part}, we know that there exist $\mathfrak h_1,\dots,\mathfrak h_N\in\Omega(x)$ and polynomials $p_1(\mathfrak d_1),\dots, p_N(\mathfrak d_N)$ with $\mathfrak h_i\mathfrak v_i = p_i(\mathfrak d_i)\mathfrak h_i$ for all $i$.
Let $n_i$ be the degree of the polynomial $p_i$ and write
$$\mathfrak h_i = \sum_{j=0}^{k_i}\partial_x^jh_{ji}(x)$$
for some rational functions $h_{0i}(x),\dots,h_{k_ii}(x)\in\bbc(x)$ with $h_{k_ii}(x)\neq 0$.

We also assumed that there exists a second-order differential operator $\mathfrak D\in\mathcal D(W)$ whose leading coefficient multiplied by $W(x)$ is Hermitian and positive-definite on the support of $W(x)$.
By Theorem \ref{backup strats} $\mathfrak T\mathfrak D = \wt{\mathfrak D}\mathfrak T$ for some second-order differential operator $\wt{\mathfrak D}\in\mathcal D(W)$.
Furthermore by the calculation in the proof of Theorem \ref{main theorem first part} we know that $E_{ii}\wt{\mathfrak D}E_{ii} = \mathfrak a_iE_{ii}$ for some second-order differential operator $\mathfrak a_i\in\bbc[\mathfrak d_i]$.
Let $C(n)$ be the sequence of matrices from the proof Theorem \ref{main theorem first part} and define a sequence of matrix polynomials $\wt P(x,n)$ and sequences of scalar polynomials $g_i(x,n)$ by
$$\wt P(x,n) = C(n)P(x,n)\cdot\wt{\mathfrak T}\ \ \text{and}\ \ g_i(x,n)E_{ii} = E_{ii}\wt P(x,n) E_{ii}.$$
Then for all $i$, we calculate
$$\wt{\mathfrak T}q(\mathfrak d_i)E_{ii} = q(\mathfrak V_i)\wt{\mathfrak T} \ \ \text{and}\ \ \wt{\mathfrak T}q(\mathfrak d_i)^2\mathfrak a_iE_{ii} = \mathfrak V_i\mathfrak D\mathfrak V_i\wt{\mathfrak T}.$$
Therefore
$$\wt P(x,n)\cdot q(\mathfrak d_i)E_{ii} = q(\lambda_i(n))E_{ii}\wt P(x,n) \ \ \text{and}\ \ \wt P(x,n)\cdot q(\mathfrak d_i)^2\mathfrak a_iE_{ii} = q(\lambda_i(n))^2\theta_i(n)E_{ii}\wt P(x,n)$$
for some polynomial $\theta_i(n)$.
Putting this together
$$\wt P(x,n)\cdot \mathfrak a_iE_{ii} = \theta_i(n)E_{ii}\wt P(x,n).$$
In particular it follows that for all $i$
$$g_i(x,n)\cdot\mathfrak a_i = \theta_i(n)g_i(x,n).$$
Note also that $\wt P(x,n)E_{ii}\cdot\mathfrak T = p_i(\lambda_i(n))E_{ii}P(x,n)$.
Defining $\mathfrak t_i$ by $\mathfrak t_iE_{ii} = E_{ii}\mathfrak TE_{ii}$, this says that for all but possibly finitely many $n$ (corresponding to roots of $p_i(\lambda_i(n))$) the function $g_i(x,n)\cdot\mathfrak t_i$ is a polynomial of degree $n$.
Hence for each $i$ there exists an integer $m_i$ (possibly negative) such that for all but finitely many $n$ the degree of $g_i(x,n)$ is $n-m_i$.
In particular $\mathfrak a_i$ is exceptional and by Theorem \ref{exceptional theorem} we know that $\mathfrak a_i$ is Darboux conjugate to a degree-filtration preserving differential operator of order $2$.

The previous paragraph shows that we may have originally chosen $\mathfrak h_i$, $\mathfrak T,\wt{\mathfrak T}$ so that the $\mathfrak a_i$ (or equivalently the $\mathfrak d_i$) are degree-filtration preserving.
Thus without loss of generality we take $\mathfrak d_i$ to be degree-filtration preserving for all $i$.
In particular for each $\mathfrak d_i$ we may write
$$\mathfrak d_i = \partial_x^2a_{2i}(x) + \partial_x a_{1i}(x) + a_{0i}(x)$$
where $a_{ji}(x)$ is a polynomial with $\deg(a_{ji}(x))\leq j$ for all $1\leq i\leq N$ and $0\leq j\leq 2$.

We next show that $\mathfrak d_i$ must be order $2$ and equal to one of the classical second-order differential operators of Hermite, Laguerre and Jacobi (hereafter referred to simply as classical operators).
Let $t_i$ be the order of $\mathfrak d_i$, which is either $1$ or $2$.
Comparing leading coefficients in the expression $\mathfrak h_i\mathfrak v_i = p_i(\mathfrak d_i)\mathfrak h_i$ we see
$$v_{m_ii}(x) = a_{t_ii}(x)^{n_i}.$$
Furthermore by comparing subleading coefficients we find
\begin{align*}
  & h_{k_ii}(x)v_{(m_i-1)i}(x) + h_{(k_i-1)i}(x)v_{m_ii}(x) + m_ih'_{k_ii}(x)v_{m_ii}(x)\\
= & a_{t_ii}(x)^{n_i}h_{(k_i-1)i}(x) + (n_ia_{t_ii}(x)^{n_i-1}a_{(t_i-1)i}(x) + (t_in_i-t_i+m_in_i)a_{t_ii}(x)^{n_i-1}a_{t_ii}'(x))h_{k_ii}(x).
\end{align*}
Combining this with the first equation, this simplifies to
$$v_{(m_i-1)i}(x) = a_{t_ii}(x)^{n_i}\left[-m_i\frac{h'_{k_ii}(x)}{h_{k_ii}(x)} + \frac{n_ia_{(t_i-1)i}(x) + (n_i(m_i+t_i)-t_i)a_{t_ii}'(x)}{a_{t_ii}(x)}\right].$$
Therefore by (\ref{equation for r real})-(\ref{equation for r imag}) we calculate that up to a constant multiple the function $r_i(x)$ is given by
\begin{equation}\label{equation for r real updated}
r_i(x) = \Re\left[a_{t_ii}(x)^{\left(\frac{2\ell_in_i}{m_i}-n_i + \frac{2n_i(m_i+t_i)-2t_i}{m_i}\right)}h_{k_ii}(x)^{-2}\exp\int \frac{2n_i}{m_i}\frac{a_{(t_i-1)i}(x)}{a_{t_ii}(x)}dx\right]
\end{equation}
for $m_i$ even and
\begin{equation}\label{equation for r imag updated}
r_i(x) = \Im\left[a_{t_ii}(x)^{\left(\frac{2\ell_in_i}{m_i}-n_i + \frac{2n_i(m_i+t_i)-2t_i}{m_i}\right)}h_{k_ii}(x)^{-2}\exp\int \frac{2n_i}{m_i}\frac{a_{(t_i-1)i}(x)}{a_{t_ii}(x)}dx\right]
\end{equation}
for $m_i$ odd.

The leading coefficient $v_{m_ii}(x)$ of $\mathfrak v_i$ must vanish at the finite endpoints of the support of $W(x)$.
Since $v_{m_ii}(x) = a_{t_ii}(x)^{n_i}$, this implies that $a_{t_ii}(x)$ must vanish at the finite endpoints of the support of $W(x)$.
If $t_i$ is $1$, then (\ref{equation for r real updated})-(\ref{equation for r imag updated}) imply that $r_i(x)$ is rational.
Thus in this case $W(x)$ must be rational, and for $W(x)$ to have finite moments the support of $W(x)$ must then be a finite interval.
However, $a_{t_ii}(x)$ can only have one root in this case, so $a_{t_ii}(x)$ does not vanish at both the endpoints of the support of $W(x)$.
Thus we see that $t_i=2$ for all $i$ and that $\mathfrak d_i$ is an operator of order $2$.

Suppose that for some $i$ the polynomial $a_{2i}(x)$ has degree two.
Then (\ref{equation for r real updated})-(\ref{equation for r imag updated}) tell us that for $W(x)$ to have finite moments the support of $W(x)$ must lie on a finite interval.
Therefore in this case the roots of $a_{2i}(x)$ must be real and distinct and equal to the finite endpoints of the support of $W(x)$.
Up to affine translation, they can be taken to be $\pm 1$ in which case $\mathfrak d_i$ is equal to a Jacobi operator.
Note that in this case since $a_{2j}(x)$ must vanish on the support of $W(x)$, they must all be Jacobi operators.

Alternatively suppose that for some $i$ the polynomial $a_{2i}(x)$ has degree $1$.
Then again to have finite moments and for $a_{2i}(x)$ to vanish at the finite endpoints of the support of $W(x)$, the support of $W(x)$ must be semi-infinite.
Up to affine translation, we can take the support to be $(0,\infty)$ in which case $\mathfrak d_i$ is a Laguerre operator.
Again, this implies that all of the $\mathfrak d_j$ are Laguerre operators.

Finally, suppose that for some $i$ the polynomial $a_{2i}(x)$ has degree $0$.
Then the weight matrix $W(x)$ must be supported on the whole real line.
Also since $\mathfrak d_i$ is exceptional the polynomial $a_{1i}(x)$ must have degree $1$ in this case (since otherwise $\mathfrak d_i$ would not have polynomial eigenfunctions of all but finitely many degrees).
Therefore up to an appropriate translation we may take $\mathfrak d_i$ to be the Hermite operator. 

In any case, $\mathfrak d_i$ is a classical operator.
For all $i$ let $f_i(x)$ be the weight function associated to the classical operator $\mathfrak d_i$, and let $h_i(x,n)$ be the associated sequence of monic orthogonal polynomials.
From the above calculation (replacing $\mathfrak d_i$ with $\mathfrak a_i$), we know that the polynomials $\wt P(x,n) := C(n)P(x,n)\cdot\wt{\mathfrak T}$ are eigenfunctions of $\mathfrak d_iE_{ii}$ for all $i$ with
$$\wt P(x,n)\cdot \mathfrak d_iE_{ii} = \theta_i(n)E_{ii}\wt P(x,n)$$
for some polynomials $\theta_1(n),\dots,\theta_N(n)$.
If $\theta_i(n)$ is nonzero, the above implies that the $i,j$'th entry $\wt P(x,n)_{ij}$ of $\wt P(x,n)$ is zero for all $j\neq i$.
Thus for all but finitely many $n$, the polynomial $\wt P(x,n)$ is diagonal.
Furthermore, for all $n$ sufficiently large the degree of the diagonal entry of $\wt P(x,n)$ is $n-m$.
Thus for all but finitely many $n$, we know that $\wt P(x,n)$ is a diagonal polynomial matrix whose $i$'th diagonal piece is $\alpha_ih_i(x,n-m)$ for some nonzero constant $\alpha_i$.
Therefore by the previous lemma we know that $\wt P(x,n)$ is a bispectral Darboux transformation of a sequence of orthogonal polynomials for a direct sum of classical weights.

\end{proof}

\begin{proof}[Proof of Theorem \ref{2x2 classification}]
We know in this case that $\mathcal D(W)$ is full and that the center of $\mathcal D(W)$ is irreducible.
The result follows immediately.
\end{proof}

\section{Examples}
In this section, we will review a handful of examples of $2\times 2$ weight matrices with $\mathcal D(W)$ full and containing a $W$-symmetric second-order differential operator whose leading coefficient multiplied by $W(x)$ is positive definite on the support of $W(x)$.
Recall that by Theorem \ref{2x2 classification} in order for $\mathcal D(W)$ to be full it need only contain two elements which do not commute.
For this reason, almost all of the $2\times 2$ weights featured in the literature whose polynomials satisfy a second order differential equation could be featured here.
However, we emphasize that our classification theorem holds for $N\times N$ weight matrices.
For examples of noncommutative bispectral Darboux transformations of weight matrices with $N>2$ see \cite{casper2015}.

For each example we consider, we know by the Classification Theorem that $W(x)$ is a bispectral Darboux transformation of a classical scalar weight $f(x)I$.
We will explicitly calculate a bispectral Darboux transformation.
The strategy for doing so mimics the proofs of the previous section.
We first determine an orthogonal system $\mathfrak V_1,\mathfrak V_2$ for $\mathcal D(W)$.
Next we determine a generator $\vec{\mathfrak u_i}\in \Omega(x)^{\oplus 2}$ for the cyclic left $\Omega(x)$-module $\mathcal M_i$ defined by (\ref{equation for M}).
Using this, we calculate $\mathfrak v_i\in\Omega(x)$ from (\ref{equation for v}).
As we proved above, the operator $\mathfrak v_i$ will be Darboux conjugate to $p_i(\mathfrak d)$ for $p_i$ a polynomial and $\mathfrak d$ one of the second-order operators associated to a classical weight $f(x)I$.
However in all of our examples below, we will see that  $\mathfrak v_i$ is \emph{already} a polynomial in $\mathfrak d$.
Thus $W(x)$ will be a noncommutative bispectral Darboux transform of $f(x)I$ with associated transformation matrix $\mathfrak U = [\vec{\mathfrak u_1}\ \vec{\mathfrak u_2}]^T$.
The calculations in this section have been verified using a python code written by the authors using Sympy for symbolic computations.

\subsection{An example of Hermite type}
Let $a\in\bbr$ and consider the weight matrix
$$W(x) = e^{-x^2}\mxx{1+a^2x^2}{ax}{ax}{1}.$$
It is known \cite{castro2006} that for this matrix $D(W)$ contains $\bbc I$ along with the span of the four second-order operators
$$\mathfrak D_1 = \partial_x^2 I + \partial_x\mxx{-2x}{2a}{0}{-2x} + \mxx{-2}{0}{0}{0},$$
$$\mathfrak D_2 = \partial_x^2\mxx{-\frac{a^2}{4}}{\frac{a^3x}{4}}{0}{0} + \partial_x\mxx{0}{\frac{a}{2}}{-\frac{a}{2}}{\frac{a^2x}{2}} + \mxx{0}{0}{0}{1},$$
$$\mathfrak D_3 = \partial_x^2\mxx{-\frac{a^2x}{2}}{\frac{a^3x^2}{2}}{-\frac{a}{2}}{\frac{a^2x}{2}} + \partial_x\mxx{-(a^2+1)}{a(a^2+2)t}{0}{1} + \mxx{0}{\frac{a^2+2}{a}}{0}{0},$$
$$\mathfrak D_4 = \partial_x^2\mxx{-\frac{a^3x}{4}}{\frac{a^2(a^2x^2-1)}{4}}{-\frac{a^2}{4}}{\frac{a^3x}{4}} + \partial\mxx{-\frac{a^3}{2}}{a^2(a^2+2)\frac{x}{2}}{0}{0} + \mxx{0}{\frac{a^2+2}{2}}{1}{0}.$$
Thus we know that $D(W)$ is full.
Furthermore $b_P^{-1}(\mathfrak D_i) = \Lambda_i(n)$ for
$$\Lambda_1(n) = \mxx{-2n-2}{0}{0}{-2n},\ \ \Lambda_2(n) = \mxx{0}{0}{0}{\frac{a^2n+2}{2}},$$
$$\Lambda_3(n) = \mxx{0}{\frac{(a^2n+2)(a^2n+a^2+2)}{2a}}{0}{0},\ \ \Lambda_4(n) = \mxx{0}{\frac{(a^2n+2)(a^2n+a^2+2)}{4}}{1}{0}.$$
Define
$$\mathfrak V_1 = \mathfrak D_2,\ \ \mathfrak V_2 = a^2\mathfrak D_1 + 4\mathfrak D_2 - 4I,$$
so that
$$\mathfrak V_2 = \partial_x^2\mxx{0}{a^3x}{0}{a^2} + \partial_x\mxx{-2a^2x}{2a^3+2a}{-2a}{0} + \mxx{-2a^2-4}{0}{0}{0},$$
$$b_P^{-1}(\mathfrak V_2) = \mxx{-2a^2n-2a^2-4}{0}{0}{0}.$$
Then $b_P^{-1}(\mathfrak V_i)b_P^{-1}(\mathfrak V_j) = 0$ for $i\neq j$ and therefore $\mathfrak V_i\mathfrak V_j = 0$ for $i\neq j$.  Moreover we may write
$$\mathfrak V_1 = \binom{\partial_x\frac{a}{2}}{1}\left(-\partial_x\frac{a}{2}\ \ \ \partial_x\frac{a^2x}{2} + 1\right),$$
$$\mathfrak V_2 = \binom{-\partial_x2a^2x-2a^2-4}{-\partial_x2a}\left(1\ \ \ -\partial_x\frac{1}{2}a\right).$$
From this, one easily finds
$$\mathcal M_1 = \{\vec{\mathfrak u}\in \Omega(x)^{\oplus 2}: \vec{\mathfrak u}^T\cdot\mathfrak V_2 = \vec 0\} = \Omega(x)\vec{\mathfrak u_1},\ \ \text{for}\ \  \vec{\mathfrak u_1} = \binom{\partial_x\frac{a}{2}}{-\partial_x\frac{a^2x}{2} - 1},$$
$$\mathcal M_2 = \{\vec{\mathfrak u}\in \Omega(x)^{\oplus 2}: \vec{\mathfrak u}^T\cdot\mathfrak V_1 = \vec 0\} = \Omega(x)\vec{\mathfrak u_2},\ \ \text{for}\ \  \vec{\mathfrak u_2} = \binom{-1}{\partial_x\frac{a}{2}}.$$
Thus we consider the operator
$$\mathfrak U = \mxx{\partial_x\frac{a}{2}}{-\partial_x\frac{a^2x}{2}-1}{-1}{\partial_x\frac{a}{2}}.$$
Its leading coefficient is
$$U(x) = \mxx{a/2}{-a^2x/2}{0}{a/2}.$$
Therefore $W(x)$ is a bispectral Darboux transformation of
$$R(x) = U(x)W(x)U(x)^* = \frac{a^2}{4}e^{x^2}I.$$
We further calculate that for $\mathfrak d = \partial_x^2 - \partial_x 2x$ the Hermite operator,
$$\mathfrak u_1^T \mathfrak V_1 = (-(a^2/4)\mathfrak d + 1)\mathfrak u_1^T,$$
$$\mathfrak u_2^T \mathfrak V_2 = (a^2\mathfrak d- 2a^2-4)\mathfrak u_2^T.$$
Thus
$$\mathfrak U \mathfrak V_1 = \mxx{-(a^2/4)\mathfrak d+1}{0}{0}{0}\mathfrak U,$$
$$\mathfrak U \mathfrak V_2 = \mxx{0}{0}{0}{a^2\mathfrak d-2a^2-4}\mathfrak U.$$

\subsection{An example of Laguerre type}
Let $a,b\in\bbr$ with $b>-1$ and consider the weight matrix
$$W(x) = x^be^{-x}1_{(0,\infty)}(x)\mxx{1 + a^2x^2}{ax}{ax}{1}.$$
Let $P(x,n)$ be the sequence of monic orthogonal polynomials for $W(x)$.
It is known \cite{castro2006} that for this matrix $D(W)$ contains $\bbc I$ as well as the second-order $W$-symmetric differential operator
$$\mathfrak D = \partial_x^2I + \partial_x\mxx{1+b-x}{2ax}{0}{1+b-x} + \mxx{-1}{a(b+1)}{0}{0}.$$
The preimage of $\mathfrak D$ under the generalized Fourier map $b_P$ is
$$b_P^{-1}(\mathfrak D) = \mxx{-n-1}{0}{0}{-n}.$$
The algebra $\mathcal D(W)$ does not contain any differential operators of order $3$, but does have a four dimensional space of operators of order $4$ modulo operators in $\mathcal D(W)$ of lower orders.
In particular $\mathcal D(W)$ contains two particularly nice $W$-symmetric differential operators of order four, which we will call $\mathfrak D_1$ and $\mathfrak D_2$ and will now describe.
The four entries of $\mathfrak D_1$ from left to right and from top to bottom are given by
\begin{align*}
&\partial_x^4a^2x^2 + \partial_x^32a^2((b+2)x-x^2)+\partial_x^2a^2((b+1)(b+2)-(3b+7)x) - \partial_xa^2(b+1)(b+3),\\
-&\partial_x^4a^3x^3 + \partial_x^3a^3x^2(2x-2(b+2)) + \partial_x^2ax(a^2(3b+7)x-a^2(b-1)(b-2)-1)\\
&\ \ + \partial_xa(a^2x(b+1)(b+3)+2x-b-1)+a(b+1),\\
-&\partial_x^2ax - \partial_xa(b+1),\\
&\partial_x^2a^2x^2 + \partial_xa^2x(b+1) + 1.
\end{align*}
% calculated with python
%[[ 1.0*D**4*a**2*t**2 + D**3*(2.0*a**2*b*t - 2.0*a**2*t**2 + 4.0*a**2*t) + D**2*(1.0*a**2*b**2 - 3.0*a**2*b*t + 3.0*a**2*b - 7.0*a**2*t + 2.0*a**2) + D*(-1.0*a**2*b**2 - 4.0*a**2*b - 3.0*a**2)
%  -1.0*D**4*a**3*t**3 + D**3*(-2.0*a**3*b*t**2 + 2.0*a**3*t**3 - 4.0*a**3*t**2) + D**2*(-1.0*a**3*b**2*t + 3.0*a**3*b*t**2 - 3.0*a**3*b*t + 7.0*a**3*t**2 - 2.0*a**3*t - 1.0*a*t) + D*(1.0*a**3*b**2*t + 4.0*a**3*b*t + 3.0*a**3*t - 1.0*a*b + 2.0*a*t - 1.0*a) + 1.0*a*b + 1.0*a]
% [-1.0*D**2*a*t + D*(-1.0*a*b - 1.0*a)
%  1.0*D**2*a**2*t**2 + D*(1.0*a**2*b*t + 1.0*a**2*t) + 1.0]]
Similarly the four entries of $\mathfrak D_2$ are given by
\begin{align*}
&\partial_x^2a^2t^2 + \partial_xa^2x(b+3) + a^2(b+1)+1,\\
&\partial_x^4a^3x^3 + \partial_x^32a^3x^2(b+4-x) + \partial_x^2ax(a^2(b+7)(b+2)+1-a^2x(3b-11))\\
&\ \ +\partial_xa(2a^2(b+1)(b+2)+b+1-(a^2(b^2+b+11)+2)x),\\
&\partial_x^2ax + \partial_xa(b+1),\\
&\partial_x^4a^2x^2 +\partial_x^32a^2x(b+2-x) + \partial_x^2a^2((b+1)(b+2)-(3b+5)x) - \partial_xa^2(b+1)^2.
\end{align*}
% calculated with python
%[[ 1.0*D**2*a**2*t**2 + D*(1.0*a**2*b*t + 3.0*a**2*t) + 1.0*a**2*b + 1.0*a**2 + 1.0
%  1.0*D**4*a**3*t**3 + D**3*(2.0*a**3*b*t**2 - 2.0*a**3*t**3 + 8.0*a**3*t**2) + D**2*(1.0*a**3*b**2*t - 3.0*a**3*b*t**2 + 9.0*a**3*b*t - 11.0*a**3*t**2 + 14.0*a**3*t + 1.0*a*t) + D*(-1.0*a**3*b**2*t + 2.0*a**3*b**2 - 8.0*a**3*b*t + 6.0*a**3*b - 11.0*a**3*t + 4.0*a**3 + 1.0*a*b - 2.0*a*t + 1.0*a) - 1.0*a**3*b**2 - 2.0*a**3*b - 1.0*a**3 - 1.0*a*b - 1.0*a]
% [1.0*D**2*a*t + D*(1.0*a*b + 1.0*a)
%  1.0*D**4*a**2*t**2 + D**3*(2.0*a**2*b*t - 2.0*a**2*t**2 + 4.0*a**2*t) + D**2*(1.0*a**2*b**2 - 3.0*a**2*b*t + 3.0*a**2*b - 5.0*a**2*t + 2.0*a**2) + D*(-1.0*a**2*b**2 - 2.0*a**2*b - 1.0*a**2)]]
The generalized Fourier map $b_P$ satisfies
$$b_P^{-1}(\mathfrak D_1) = \mxx{0}{a+(a^3b^2+a^3b+2a)n+(3a^3b+a^3)n^2+2a^3n^3}{0}{1+a^2bn+a^2n^2},$$
%[[0
%  1.0*a**3*b**2*t + 3.0*a**3*b*t**2 + 1.0*a**3*b*t + 2.0*a**3*t**3 + 1.0*a**3*t**2 + 1.0*a*b + 2.0*a*t + 1.0*a]
% [0 1.0*a**2*b*t + 1.0*a**2*t**2 + 1.0]]
$$b_P^{-1}(\mathfrak D_2) = \mxx{a^2b+a^2+1+(a^2b+2a^2)n+a^2n^2}{0}{-a^3(b+1)^2-a(b+1)-a(a^2(b+4)(b+1)+2)n-a^3(3b+5)n^2-2a^3n^3}{0}^T
$$
%[[1.0*a**2*b*t + 1.0*a**2*b + 1.0*a**2*t**2 + 2.0*a**2*t + 1.0*a**2 + 1.0
%  -1.0*a**3*b**2*t - 1.0*a**3*b**2 - 3.0*a**3*b*t**2 - 5.0*a**3*b*t - 2.0*a**3*b - 2.0*a**3*t**3 - 5.0*a**3*t**2 - 4.0*a**3*t - 1.0*a**3 - 1.0*a*b - 2.0*a*t - 1.0*a]
% [0 0]]
Note that the second expression has a transpose, in order to better fit within the page.
It is clear from the eigenvalue expressions that $b_P^{-1}(\mathfrak D_i)b_P^{-1}(\mathfrak D_j)=0$ for $i\neq j$ 
and that $b_P^{-1}(\mathfrak D_i)$ commutes with $b_P^{-1}(\mathfrak D)$ for all $i$.  Thus $\mathfrak D_i\mathfrak D_j=0$ 
for $i\neq j$ and $\mathfrak D_i$ commutes with $\mathfrak D$ for all $i$.
Then defining
$$\mathfrak V_1 = \mathfrak D_1 ((a^2b+a^2+1)I-(a^2b+2a^2)\mathfrak D+a^2\mathfrak D^2)$$
and
$$\mathfrak V_2 = \mathfrak D_2 (I-a^2b(\mathfrak D-I)+a^2(\mathfrak D-I)^2),$$
we get that $\mathfrak V_i$ is $W$-symmetric for all $i$ and that $\mathfrak V_i\mathfrak V_j=0$ for $i\neq j$.
We also find
$$b_P(\mathfrak V_1+\mathfrak V_2) = (1+a^2bn+a^2n^2)(a^2b+a^2+1+(a^2b+2a^2)n+a^2n^2)I,$$
so $\mathfrak V_1+\mathfrak V_2\in\mathcal Z(W)$.
Thus $\mathfrak V_1$ and $\mathfrak V_2$ for an orthogonal system for $\mathcal D(W)$.

We calculate
$$\mathcal M_i := \{\vec{\mathfrak u}^T\in \Omega(x)^{\oplus 2}: \vec{\mathfrak u}^T\mathfrak V_j=\vec 0^T,\ j\neq i\} = \Omega(x)\mathfrak u_i$$
for
$$\vec{\mathfrak u_1} = \binom{\partial_x^2ax + \partial_xa(b+1)}{-\partial_x^2a^2x^2 -\partial_xa^2(b+1)x-1},\ \ \text{and}\ \ \vec{\mathfrak u_2} = \binom{1}{\partial_x^2ax + \partial_xa(b+1-2x) -a(b+1)}.$$
%[[D**2*a*t + D*(a*b + a) -D**2*a**2*t**2 + D*(-a**2*b*t - a**2*t) - 1]
% [1 D**2*a*t + D*(a*b - 2*a*t + a) - a*b - a]]
Thus we define
$$\mathfrak U = \mxx{\partial_x^2ax + \partial_xa(b+1)}{-\partial_x^2a^2x^2 -\partial_xa^2(b+1)x-1}{1}{\partial_x^2ax + \partial_xa(b+1-2x) -a(b+1)},$$
and we calculate
$$\mathfrak U\mathfrak D = \mxx{\mathfrak d}{0}{0}{\mathfrak d-1}\mathfrak U$$
where here $\mathfrak d = \partial_x^2x + \partial_x(b+1-x)$ is the Laguerre operator for the Laguerre weight $f(x) = x^be^{-x}1_{(0,\infty)}(x)$.
In this way $W(x)$ is a bispectral Darboux transformation of the Laguerre weight.
\subsection{An example of Jacobi type}
Let $r,a\in\bbr$ with $r>0$ and consider the weight matrix
$$W(x) = (1-x^2)^{r/2-1}1_{(-1,1)}(x)\mxx{a(x^2-1)+r}{-rx}{-rx}{(r-a)(x^2-1)+r}.$$
Let $P(x,n)$ be the sequence of monic orthogonal polynomials for $W(x)$.
It is known \cite{zurrian2016algebra} that for this matrix $D(W)$ contains $\bbc I$ as well as the four second-order $W$-symmetric differential operators
$$\mathfrak D_1 = \partial_x^2\mxx{x^2}{x}{-x}{1} + \partial_x\mxx{(r+2)x}{r-a+2}{-a}{0} + \mxx{a(r-a+1)}{0}{0}{0},$$
$$\mathfrak D_2 = \partial_x^2\mxx{-1}{-x}{x}{x^2} + \partial_x\mxx{0}{a-r}{a+2}{(r+2)x} + \mxx{0}{0}{0}{(a+1)(r-a)},$$
$$\mathfrak D_3 = \partial_x^2\mxx{-x}{-1}{x^2}{x} + \partial_x\mxx{-a}{0}{2(a+1)x}{a+2} + \mxx{0}{0}{a(a+1)}{0},$$
$$\mathfrak D_4 = \partial_x^2\mxx{x}{x^2}{-1}{-x} + \partial_x\mxx{r-a+2}{2(r-a+1)x}{0}{a-r} + \mxx{0}{(r-a)(r-a+1)}{0}{0}.$$
The preimages of these elements under the generalized Fourier map are given by
$$b_P^{-1}(\mathfrak D_1) = \mxx{(n+a)(n+r-a+1)}{0}{0}{0},\ \ b_P^{-1}(\mathfrak D_2) = \mxx{0}{0}{0}{(n+a+1)(n+r-a)},$$
$$b_P^{-1}(\mathfrak D_3) = \mxx{0}{0}{(n+a)(n+a+1)}{0},\ \ b_P^{-1}(\mathfrak D_4) = \mxx{0}{(n+r-a)(n+r-a+1)}{0}{0}.$$
The operators $\mathfrak D_1$ and $\mathfrak D_2$ are $W$-symmetric and satisfy $\mathfrak D_i\mathfrak D_j = 0$ for $i\neq j$.
Define
$$\mathfrak V_1 = (\mathfrak D_1 + (r-2a)I)\mathfrak D_1,$$
$$\mathfrak V_2 = (\mathfrak D_2 - (r-2a)I)\mathfrak D_2.$$
Then $\mathfrak V_i$ is $W$-symmetric for all $i$ and $\mathfrak V_i\mathfrak V_j = 0$ for $i\neq j$.
Furthermore
$$b_P^{-1}(\mathfrak V_1 + \mathfrak V_2) = (n+a)(n+r+1-a)[(n+a)(n+r+1-a)+r-2a]I,$$
and so $\mathfrak V_1+\mathfrak V_2\in\mathcal Z(W)$.
Thus $\mathfrak V_1$ and $\mathfrak V_2$ form an orthogonal system for $\mathcal D(W)$.

We calculate 
$$\mathcal M_1 = \{\vec{\mathfrak u}\in \Omega(x)^{\oplus 2}: \vec{\mathfrak u}^T\cdot\mathfrak V_2 = \vec 0\} = \Omega(x)\vec{\mathfrak u_1},\ \ \text{for}\ \  \vec{\mathfrak u_1} = \binom{\partial_xx + a}{\partial_x},$$
$$\mathcal M_2 = \{\vec{\mathfrak u}\in \Omega(x)^{\oplus 2}: \vec{\mathfrak u}^T\cdot\mathfrak V_1 = \vec 0\} = \Omega(x)\vec{\mathfrak u_2},\ \ \text{for}\ \  \vec{\mathfrak u_2} = \binom{\partial_x}{\partial_xx + r-a}.$$
Then defining
$$\mathfrak U = \mxx{\partial_xx + a}{\partial_x}{\partial_x}{\partial_xx + r-a},$$
we calculate
$$\mathfrak U\mathfrak V_1 = \mxx{-\mathfrak d + (a+1)r-a(a+1)}{0}{0}{0}\mathfrak U,$$
$$\mathfrak U\mathfrak V_2 = \mxx{0}{0}{0}{-\mathfrak d + (1+a)(r-a)}\mathfrak U$$
where here $\mathfrak d= \partial_x^2(1-x^2) - \partial_x(r+2)x$ is the Jacobi (or Gegenbauer) operator associated with the classical weight $f(x) = (1-x^2)^{r/2}1_{(-1,1)}(x)$.
Thus $W(x)$ is a bispectral Darboux transformation of $f(x)I$.
\section*{Acknowledgement} We are grateful to Ken Goodearl and F. Alberto Gr\"unbaum for very helpful correspondence on $*$-algebras and 
comments on the first version of the paper.
The research of the authors was supported by NSF grant DMS-1601862 and Bulgarian Science Fund grant H02/15.
\bibliographystyle{plain}
\bibliography{full}

\end{document}